\documentclass[12pt]{amsart}
\usepackage[margin=1in]{geometry}
\usepackage{latexsym}
\usepackage{float}
\usepackage{amssymb}
\usepackage[all]{xy}
\usepackage{epsfig}
\usepackage{color}
\usepackage{graphics}
\usepackage{hyperref}

\newtheorem{theorem}{Theorem}[section]

\newtheorem{Thm}[theorem]{Theorem}
\newtheorem{Lem}[theorem]{Lemma}
\newtheorem{Prop}[theorem]{Proposition}
\newtheorem{Cor}[theorem]{Corollary}
\newtheorem{Def}[theorem]{Definition}

\newcommand{\al}{\alpha}
\newcommand{\be}{\beta}

\newcommand{\ep}{\epsilon}
\newcommand{\del}{\delta}

\begin{document}

\title[The rectangle condition does not detect the strong irreducibility]{The rectangle condition does not detect the strong irreducibility}
\author{Bo-hyun Kwon, Sungmo Kang, and Jung Hoon Lee}

\address{Department of Mathematical Sciences, UNIST, 50 UNIST-gil, Ulsan 44919, Republic of Korea}\email{bortire74@gmail.com}
\address{Department of Mathematics Education, Chonnam National University, 77 Yongbong-ro, Buk-gu, Gwangju 61186, Republic of Korea}\email{skang4450@chonnam.ac.kr}
\address{Department of Mathematics and Institute of Pure and Applied Mathematics, Jeonbuk National University, 567 Baekje-daero, Deokjin-gu, Jeonju, Jeonbuk, 54896, Republic of Korea}\email{junghoon@jbnu.ac.kr}
%\date{\today}

\thanks{The first author is supported by Basic Science Research Program through the National Research Foundation of Korea(NRF) funded by the Ministry of Education(RS-2023-00242412), the second author is supported by the National Research Foundation of Korea(NRF) grant
funded by the Korea government(MSIT)(2021R1F1A1063187), and the third author is supported by the National Research Foundation of Korea(NRF) grant funded by the Korea government(MSIT) (RS-2023-00275419).}


\maketitle

\begin{abstract}
The rectangle condition for a genus $g$ Heegaard splitting of a 3-manifold, defined by Casson and Gordon in \cite{CG}, provides a sufficient criterion for the Heegaard splitting to be strongly irreducible. However it is unknown whether there exists a strongly irreducible Heegaard splitting which does not satisfy the rectangle condition. In this paper we provide a counterexample of a genus 2 Heegaard splitting of a 3-manifold which is strongly irreducible but fails to satisfy the rectangle condition. The way of constructing such an example is to take a double branched cover of a 3-bridge decomposition of a knot in $S^3$ which is strongly irreducible but does not meet the rectangle condition. This implies that the rectangle condition does not detect the strong irreducibility. As our next goal, we expect that this result provides the weaker version of the rectangle condition which detects the strong irreducibility.
\end{abstract}

\section{Introduction and main results}\label{B1}

A genus $g$ Heegaard splitting of a 3-manifold $M$ is defined to be a triple $(V, W; \Sigma)$, where $\Sigma$ is a genus $g$ Heegaard surface that
decomposes $M$ into two handlebodies $V$ and $W$.
The splitting is said to be \textit{strongly irreducible} if every pair of an essential disk in $V$ and an essential disk in $W$ intersects,
otherwise it is weakly reducible. The notion of the strong irreducibility provides much information on a 3-manifold.
For instance, if a 3-manifold admits a strongly irreducible Heegaard splitting, then it is irreducible and $\partial$-irreducible.
Furthermore, if a genus 2 Heegaard splitting is strongly irreducible, then the Heegaard genus of the 3-manifold is 2.
This is because if a 3-manifold $M$ has a Heegaard genus 0 or 1, then $M$ is either $S^3$ or a lens space,
in which case any genus $n$($n\geq 2$) Heegaard splitting of $S^3$ or a lens space is weakly reducible.

Casson and Gordon \cite{CG} provide a nicely sufficient condition for a Heegaard splitting to be strongly irreducible, called
a rectangle condition. The rectangle condition of a Heegaard splitting $(V, W; \Sigma)$ of a 3-manifold $M$ is defined to be a certain condition
requiring nine types of rectangles on the intersections of two pants decompositions of $\Sigma$ which are obtained from $V$ and $W$ by deleting the maximal collections of
nonisotopic pairwise disjoint essential disks in $V$ and in $W$ respectively. The details of the rectangle condition are given in Section~\ref{B3}.
Casson and Gordon \cite{CG} proved the following.

\begin{Thm}[\cite{CG}]\label{Casson-Gordon}
If a Heegaard splitting of a 3-manifold satisfies the rectangle condition, then it is strongly irreducible.
\end{Thm}

Consequently, for the case of a genus 2 Heegaard splitting, if a genus 2 Heegaard splitting of a 3-manifold satisfies the rectangle condition, then its Heegaard
genus is 2.

The notions and properties of the strong irreducibility and the rectangle condition of a Heegaard splitting of a 3-manifold can be transferred to a bridge decomposition of a knot in $S^3$.
An $n$\textit{-bridge decomposition} of a knot $K$ is defined to be a triple $(T_1, T_2;S)$, where $S$ is a bridge 2-sphere in $S^3$ which decomposes the pair $(S^3, K)$ into two rational $n$-tangles $T_1=(B_1, \tau_1)$ and $T_2=(B_2, \tau_2)$. Then a knot $K$ is an $n$-\textit{bridge knot} if it has an $n$-bridge decomposition but no $(n-1)$-bridge decomposition.
An $n$-bridge decomposition $(T_1, T_2;S)$ of a knot $K$ is strongly irreducible if every pair of a compression in $T_1$ and a compression in $T_2$ intersects.
A compression in a rational tangle $T=(B^3, \tau)$ is an essential disk whose boundary does not bound a disk nor a once-punctured disk in $\partial B-\tau$.
In \cite{K20} the first author defined the rectangle condition of an $n$\textit{-bridge decomposition} $(T_1, T_2;S)$ of a knot $K$ as
a condition on the intersections of two pants decompositions of $S$ which are obtained from $T_1$ and $T_2$ by deleting the maximal collections of
nonisotopic pairwise disjoint compressions in $T_1$ and $T_2$ respectively.
The first author \cite{K20} showed that the rectangle condition of a bridge decomposition of a knot in $S^3$ guarantees the strong irreducibility.

\begin{Thm}[\cite{K20}]\label{Kwon2}
If an $n$-bridge decomposition of a knot in $S^3$ satisfies the rectangle condition, then it is strongly irreducible.
\end{Thm}

As a corollary of Theorem~\ref{Kwon2} in \cite{K20}, for the case of 3-bridge decomposition of a knot, the first author showed that
if a $3$-bridge decomposition of a knot satisfies the rectangle condition, then the knot is a 3-bridge knot.
Using this result, in \cite{KK21} and \cite{KK23}, the first and the second authors constructed infinitely many families of alternating 3-bridge prime knots which have 3-bridge decompositions satisfying the rectangle condition.

It is natural to investigate how closely the rectangle condition is equivalent to the strong irreducibility.
Along this line, in \cite{KL20}, the first and third authors show that the rectangle condition consisting of nine types of rectangles in some combination is sharp by showing that there exists a 3-bridge decomposition which has eight types of rectangles and also is not strongly irreducible.

\begin{Thm}[\cite{KL20}]\label{KL20}
There exist 3-bridge decompositions of knots in $S^3$ which admit a bridge diagram with eight types of rectangles but are not strongly irreducible.
\end{Thm}

The goal of this paper is to show that the converses of Theorem~\ref{Casson-Gordon} and Theorem~\ref{Kwon2} are not true.
Since the definition of the rectangle condition is so complicated, one might naturally expect the converse to fail.
However it is still unknown whether or not there is a counterexample disproving the converse of Theorem~\ref{Casson-Gordon} and Theorem~\ref{Kwon2}.
Actually it is not easy to verify that a Heegaard splitting of a 3-manifold or a bridge decomposition
of a knot doesn't satisfy the rectangle condition because we need to consider all of the pants decompositions of the Heegaard splitting or a bridge decomposition.

To attain our goal that the converses of Theorem~\ref{Casson-Gordon} and Theorem~\ref{Kwon2} are not true, first we show that both definitions of the strong irreducibility and the rectangle condition in a Heegaard splitting of a 3-manifold and in a bridge decomposition
of a knot are equivalent. As described above, there are many similarities between the theories of Heegaard splittings of 3-manifolds and bridge decompositions
of knots such as the strong irreducibility and the rectangle condition. In fact, they have a close relationship in the sense that the double branched cover of an $n$-bridge decomposition of a knot is
a genus $(n-1)$ Heegaard splitting of a 3-manifold. As our some results of this paper, we have the following theorems which are given in Section~\ref{B3} as Theorems~\ref{main result-intro-1} and ~\ref{main result-intro-2}.

\begin{Thm}\label{main result-2}
Let $(T_1, T_2; S)$ be a 3-bridge decomposition of a knot $K$ in $S^3$ and
$(V, W;\Sigma)$ the genus 2 Heegaard splitting obtained by the double branched covering of
$(T_1, T_2; S)$. Then $(V, W;\Sigma)$ is strongly irreducible if and only if $(T_1, T_2; S)$
is strongly irreducible.
\end{Thm}

\begin{Thm}\label{main result-1}
Let $(T_1, T_2; S)$ be a 3-bridge decomposition of a knot $K$ in $S^3$ and
$(V, W;\Sigma)$ the genus 2 Heegaard splitting obtained by the double branched covering of
$(T_1, T_2; S)$. Then $(V, W;\Sigma)$ satisfies the rectangle condition if and only if $(T_1, T_2; S)$ satisfies the rectangle condition.
\end{Thm}

Next, we present a counterexample of a 3-bridge decomposition $(T_1, T_2; S)$ of a knot $K$ which is strongly irreducible but does not satisfy the rectangle condition.
It follows from Theorems~\ref{main result-2} and ~\ref{main result-1} that the double branched cover of $(T_1, T_2; S)$ of a knot $K$
yields a genus 2 Heegaard splitting of a 3-manifold which is strongly irreducible but does not satisfy the rectangle condition.
The following are the main results of this paper which follow from Theorems~\ref{main result3} and ~\ref{main result4} in Section~\ref{B4}.

\begin{Thm}\label{main result}
There exists a 3-bridge decomposition of a prime knot in $S^3$ which is strongly irreducible but does not satisfy the rectangle condition.
\end{Thm}

\begin{Thm}\label{main result2}
There exists a genus $2$ Heegaard splitting of a 3-manifold which is strongly irreducible but does not satisfy the rectangle condition.
\end{Thm}

Theorems~\ref{main result} and ~\ref{main result2} imply the rectangle condition doesn't detect strong irreducibility.
Actually the knot in Theorem~\ref{main result} is the knot $8_5$ in the Rolfsen knot table of \cite{R76}.
The rectangle condition of an $n$-bridge decomposition $(T_1, T_2;S)$ of a knot $K$ consists of nine rectangles created from the intersections of two pants decompositions of $S$. However, since there are infinitely many pants decompositions of a rational 3-tangle,
in order to show that a 3-bridge decomposition of a knot does not satisfy the rectangle condition, one must verify that all of the pairs of two pants decompositions of $S$ from $T_1$ and $T_2$ fail to satisfy the rectangle condition.

Even though the definition of the rectangle condition is technical to a certain extent so that it doesn't detect the strong irreducibility, we are convinced that the rectangle condition is not far from the strong irreducibility. Thus our ongoing work is to find a weaker version of the rectangle condition which detects the strong irreducibility.
By investigating the counterexamples in Theorems~\ref{main result} and ~\ref{main result2}, we have found the notion of the normal form which is defined originally in \cite{KL24} to be essential potentially for a weaker version of the rectangle condition. We expect that the weaker version might be obtained from the minimal intersections formed by a normal form on each boundary of two pants decompositions, which is simpler to handle than rectangles in the rectangle condition.

\section{Preliminary}\label{B2}

\subsection{Rational tangles}\label{B2-1}\hfill
%\smallskip
A rational $3$-\textit{tangle} $T$ is a pair $(B, \tau)$, where $B$ is a 3-ball and $\tau$ is a set of three pairwise disjoint arcs $\{\tau_1, \tau_2, \tau_3\}$ properly embedded in $B$ such that there exists a homeomorphism of pairs ${H}: (B,\tau)\rightarrow (D^2\times I,\{p_1,p_2, p_3\}\times I)$, where $I=[0,1]$ and $p_i\in \text{int}\hskip 1pt D^2$, $1\leq i\leq 3$.

For a rational $3$-tangle $T=(B, \tau)$ with $\tau=\{\tau_1, \tau_2, \tau_3\}$, a \textit{compression} is a properly embedded disk in $B-\tau$ whose boundary is essential in $\partial B-\tau$, that is, does not bound a disk nor a once-punctured disk in $\partial B-\tau$. Note that since $(B, \tau)$ is a rational 3- tangle,
the boundary of a compression bounds a 2-punctured disk in $\partial B-\tau$. Also
note that there exists a set of three pairwise disjoint non-isotopic compressions, which is called a \textit{system of compressions} for $T$.

A disk $D$ in $B$ is called a \textit{bridge disk} if
$\text{int}D\cap \tau=\varnothing$ and $\partial D=\tau_i\cup \alpha$ for some $i$ ($1\leq i \leq 3$), where $\alpha$ is
a simple arc in $\partial B$. Such an arc $\alpha$ is called a \textit{bridge arc} for $\tau_i$.
For each $i$, there exists a bridge disk $D_i$ with a corresponding bridge arc $\alpha_i$. Note that for a tangle arc $\tau_i$,
there are infinitely many bridge disks and bridge arcs. See Lemma~\ref{infinitely many bridge arcs}.

Since $T=(B, \tau)$ is a rational 3-tangle,
there is a collection of three disjoint bridge disks $\{D_1, D_2, D_3\}$. Such a collection of disjoint bridge disks $\{D_1, D_2, D_3\}$ is
called a \textit{system of bridge disks} for $T=(B, \tau)$. The collection of bridge arcs $\{\alpha_1, \alpha_2, \alpha_3\}$ of a system of bridge disks, which is a collection of three disjoint bridge arcs, is called \textit{a system of bridge arcs} for $T$.
Note that the boundary of a regular neighborhood of a bridge disk $D_i$ in $B$ is the union
of two disks $E_i$ and $E'_i$ glued along their boundaries, where $E_i$ is a compression and $E'_i$ is a subset of $\partial B$ containing a bridge arc $\alpha_i$.
This yields one-to-one correspondences between a system of compressions, a system of bridge disks, and a system of bridge arcs.

A system of bridge arcs for the fixed trivial rational 3-tangle $T=(B, \tau)$ is said to be \textit{trivial} and is denoted
by $\ep=\{\ep_1, \ep_2, \ep_3\}$. Figure~\ref{Rectangle-1} illustrates a trivial rational 3-tangle $T$ which
contains bridge disks $D_i$'s, bridge arcs $\ep_i$'s, compressions $E_i$'s, and disks $E'_i$ in $\partial B$ containing the bridge arcs $\ep_i$ for $i=1,2,3$.

\begin{figure}[tbp]
\centering
\includegraphics[width = 0.34\textwidth]{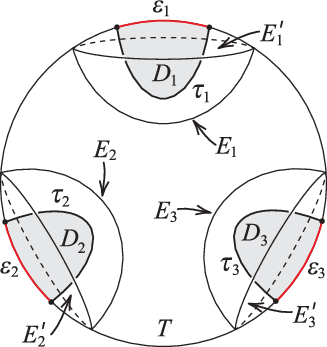}
\caption{A trivial rational 3-tangle $T$ which
contains bridge disks $D_i$'s, bridge arcs $\ep_i$'s, compressions $E_i$'s, and disks $E'_i$ in $\partial B$ containing the bridge arcs $\ep_i$ for $i=1,2,3$.
}
\label{Rectangle-1}
\end{figure}

The following lemma implies that there are infinitely many systems of bridge arcs for $T=(B, \tau)$ up to isotopy.

\begin{Lem}\label{infinitely many bridge arcs}
Let $T=(B, \tau)$ with $\tau=\{\tau_1, \tau_2, \tau_3\}$ be a rational 3-tangle. Let $\al_1$ and $\al_2$ be two disjoint bridge arcs
for $\tau_1$ and $\tau_2$ respectively. Then any arc in $\partial B$ connecting the two endpoints of $\tau_3$ and disjoint from $\al_1\cup\al_2$
is a bridge arc for $\tau_3$.
\end{Lem}

\begin{proof}
This is Lemma 1 in \cite{K24}.
\end{proof}

\begin{Cor}\label{infinitely many systems of bridge arcs}
There are infinitely many systems of bridge arcs for $T=(B, \tau)$ up to isotopy, and hence infinitely many systems of compressions and infinitely many systems of bridge disks for $T=(B, \tau)$ as well.
\end{Cor}

\begin{figure}[tbp]
\centering
\includegraphics[width = 0.34\textwidth]{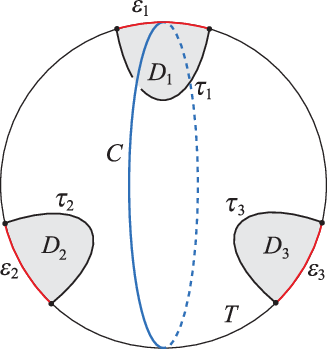}
\caption{Any arc which is obtained by twisting $n$ times ($n\in \mathbb{Z}$) the bridge arc $\ep_1$ along the circle $C$ becomes a bridge arc for $T$.}
\label{Rectangle-1-1}
\end{figure}

More explicitly, in Figure~\ref{Rectangle-1-1}, any arc obtained by twisting $n$ times ($n\in \mathbb{Z}$) the bridge arc $\ep_1$ along the circle $C$ is a bridge arc for $T$. Since there are infinitely many systems of bridge arcs in $T=(B, \tau)$, we assume that systems of bridge arcs intersect each other transversely and minimally.

Now consider a knot $K$ in $S^3$ and its bridge decomposition. An $n$\textit{-bridge decomposition} of $K$ is defined to be a triple $(T_1, T_2;S)$, where $S$ is a 2-sphere in $S^3$ that decomposes the pair $(S^3, K)$ into two rational $n$-tangles $T_1=(B_1, \tau_1)$ and $T_2=(B_2, \tau_2)$. We call $S$ a \textit{bridge sphere}. Then a knot $K$ is an $n$-\textit{bridge knot} if it has an $n$-bridge decomposition but no $(n-1)$-bridge decomposition.

It is not easy to determine whether a knot $K$ with an $n$-bridge decomposition has no $(n-1)$-bridge decomposition.
However, there is a sufficient condition for $K$ to be an $n$-bridge knot, known as the rectangle condition.
Since we are mainly interested in 3-bridge knots, from now on we will put our attention to 3-bridge decompositions of knots.\\

\textbf{Notations.} As with systems of bridge disks and systems of bridge arcs, in this paper we usually denote a set of disks $\{D_1, D_2, \ldots, D_n\}$ by $\mathcal{D}$ and a set of arcs $\{\al_1, \al_2, \ldots, \al_m\}$ by $\al$. However, for simplicity for notations, we also allow $\mathcal{D}$ to denote the union $\bigcup_{i=1}^n D_i$ and $\al$ the union $\bigcup_{j=1}^m \al_j$. \\

%여기에서부터
\subsection{Waves and Normal forms of systems of bridge arcs}\label{B2-2}\hfill
%\smallskip

In this subsection we introduce the notions of a wave and of a normal form for a system of bridge arcs in a rational 3-tangle.
These concepts play an essential role throughout the paper.

We begin with the definition of a wave for a system of bridge arcs. Before doing so, we recall the notion of a wave in two related contexts: for a system of essential disks in a handlebody, and for a system of compressions in a rational tangle, as defined in \cite{KL12} and \cite{K20}, respectively.

Let $H$ be a genus $g$ handlebody. There is a maximal collection of $(3g-3)$ mutually disjoint, non-isotopic essential disks $\{D_1, \ldots , D_{3g-3}\}$
which cuts $H$ into a collection of $(2g-2)$ balls, and correspondingly cuts $\partial H$ into a collection of $2g-2$ pants.
Let $D$ be an essential disk that is not isotopic to any of $D_i$'s ($i=1, 2, \ldots, 3g-3$). Then $D$ must intersect
$\bigcup_{i=1}^{3g-3}D_i$. We assume that $D$ intersects $\bigcup_{i=1}^{3g-3}D_i$ transversally and minimally.
Consider the intersection $D\cap(\bigcup_{i=1}^{3g-3}D_i)$. Then there exists a subarc of $\partial D$
which cobounds a disk with an outermost arc of $D$ cut off by these intersections.
Such a subarc of $\partial D$ is called a \textit{wave}, denoted by $\omega$. See Figure~\ref{Rectangle-17}a.
If the outermost arc corresponding to a wave lies in $D_{i_0}$ for some $1\leq i_0\leq 3g-3$,
then we say that \textit{the wave is based at} $D_{i_0}$.

\begin{figure}[tbp]
\centering
\includegraphics[width = 1\textwidth]{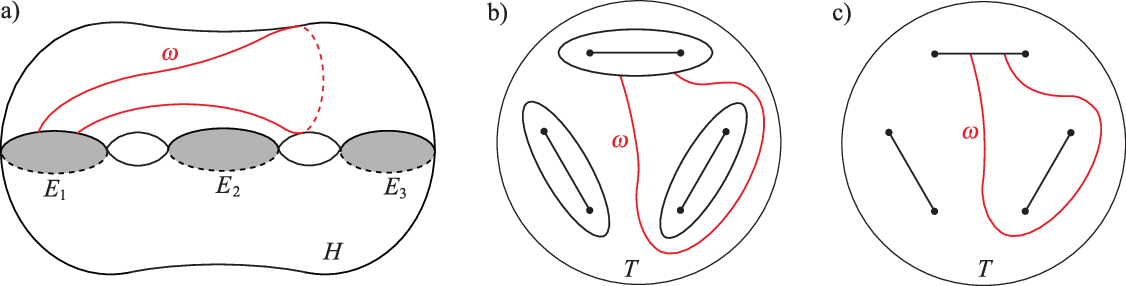}
\caption{Waves from essential disks on a genus 2 handlebody, from compressions, and from bridge arcs on a rational 3-tangle.}
\label{Rectangle-17}
\end{figure}

In a manner analogous to the case of essential disks in a handlebody, one may define a wave for a system of compressions in a rational tangle.
Let $T=(B, \tau)$ be a rational $n$-tangle, where $\tau$ is a union of $n$ strings $\tau_1, \ldots, \tau_n$.
As defined in a rational 3-tangle, a compression in $T$ is a properly embedded disk in $B$ whose boundary
is neither a disk nor a once-punctured disk in $\partial B - \tau$.
Then there exists a maximal collection of $(2n-3)$ pairwise disjoint, nonisotopic compressions $\{D_1, \ldots , D_{2n-3}\}$ in $T$,
called a system of compressions for $T$, such a system cuts $T$ into a collection of $(2n-2)$ 3-balls, while simultaneously decomposing $\partial T$ into a collection of $(n-2)$ pairs of pants together with $n$ disks.
Let $D$ be a compression in $T$ which is not isotopic to any of $D_i$ ($i=1,2,\ldots,2n-3$). Then $D$ must intersect
$\bigcup_{i=1}^{2n-3} D_i$. Assuming that $D$ intersects $\bigcup_{i=1}^{2n-3} D_i$ transversally and minimally, we consider
the intersection $D \cap \bigcup_{i=1}^{2n-3} D_i$. Then there exists a subarc of $\partial D$
that cobounds a disk with an outermost arc in $D$ arising from the intersections of $D$ and $\bigcup_{i=1}^{2n-3} D_i$.
Such a subarc of $\partial D$ is called a \textit{wave} (see Figure~\ref{Rectangle-17}b). If the outermost arc for a wave belongs to $D_{i_0}$ ($1 \leq i_0 \leq 2n-3$),
then we say that \textit{the wave is based at} $D_{i_0}$.

Now we define a wave for a system of bridge arcs in a rational 3-tangle.
Let $T=(B, \tau)$ be a $3$-rational tangle, where $\tau$ is the union of three strings $\tau_1, \tau_2, \tau_3$.
Also let $\alpha=\{\alpha_1,\alpha_2, \alpha_3\}$ be a system of bridge arcs of $T$.

A \textit{wave} with respect to a system of bridge
arcs $\al$ is defined to be an embedded arc $\omega$ in $\partial B$ such that the interior is disjoint from $\al$ and both endpoints
of $\omega$ lies in the same bridge arc $\al_i$ for some $i\in\{1, 2, 3\}$, and when
$\al_i$ and $\omega$ oriented, both endpoints(i.e., intersection points between $\al_i$ and $\omega$)
have opposite signs, and $\omega$ is essential in $\partial B -N(\al)$.
Here $N(\al)$ is a regular neighborhood of $\al$ in $\partial B$. We say that \textit{a wave} $\omega$ \textit{is based at} $\al_i$.
An example of a wave based at $\alpha_i$ is illustrated in Figure~\ref{Rectangle-17}c.
We note from Figure~\ref{Rectangle-17}c that a wave $\omega$ based at $\al_i$ in a rational 3-tangle separates the two remaining bridge arcs
$\al_j$ and $\al_k$. We also observe that the condition in the definition of a wave that both endpoints(i.e., the intersection points between $\al_i$ and $\omega$)
have opposite signs follows from the corresponding property of waves for a system of compressions and for a collection
of essential disks in a handlebody.

The following lemma establishes the existence of a wave associated with a bridge arc.

\begin{Lem}\label{wave existence}
Suppose  $\alpha=\{\alpha_1,\alpha_2, \alpha_3\}$ is a system of bridge arcs for a rational 3-tangle $T$. If a bridge arc $\gamma$ in $T$, which
is not isotopic to any of $\alpha_i$ ($i=1, 2, 3)$, intersects $\al$, then $\gamma$ admits a wave with respect to $\alpha$.
\end{Lem}

\begin{proof}
We recall that the boundary of a regular neighborhood of a bridge disk $D_i$ in $B$ is the union
of two disks $E_i$ and $E'_i$ glued along their boundaries, where $E_i$ is a compression and $E'_i$ is a subset of $\partial B$ containing a bridge arc $\alpha_i$ for $F$.
Consequently, the boundary of a compression $E_i$ can be identified with a regular neighborhood of the bridge arc $\alpha_i$ in $\partial B$.

If a bridge arc $\gamma$ does not intersect $\alpha$ in its interior, then since it is not isotopic to any of $\alpha_i$ ($i=1, 2, 3)$, the arc $\gamma$ itself forms a wave.
Thus, we may assume that $\gamma$ intersects $\alpha$ in its interior. For each $i\in\{1, 2, 3\}$ let $D_i$ be a bridge disk for $\alpha_i$ and let $E_i$ denote the corresponding compression. Let $D$ be a bridge disk for $\gamma$, and let $E$ be the compression associated to $D$.
Since $\gamma$ intersects $\alpha$ in its interior, the compression $E$ necessarily intersects the system of compressions $\{E_1, E_2, E_3\}$ in $T$.
Then, as established in the definition of a wave for a system of compressions in a rational tangle, there exists a wave based at some $E_{i_0}$, which
is realized as a subarc of $\partial E$ whose endpoints lie on $\partial E_{i_0}$. Since, in general, the boundary of a compression can be identified with a regular neighborhood of a bridge arc in $\partial B$, it follows that there exists a subarc $\omega$ of $\gamma$ whose endpoints both lie on the same bridge arc $\al_{i_0}$. This subarc $\omega$ is therefore a wave.
\end{proof}

Now we introduce the notion of a normal form between two systems of bridge arcs, originally defined by the first author in \cite{K24}.
Let $\alpha=\{\alpha_1,\alpha_2,\alpha_3\}$ and $\beta=\{\beta_1,\beta_2,\beta_3\}$ be systems of bridge arcs for a rational 3-tangle $T = (B,\tau)$.
We say that the bridge arc system $\alpha$ is \textit{in a normal form} with respect to $\beta$ if, for each $\beta_i$, there exist no two adjacent
intersections between $\alpha$ and $\beta_i$ in $\beta_i$ which belong to the same $\alpha_j$ for some $j\in\{1, 2, 3\}$, and vice versa.

There are several properties of the normal form for the systems of bridge arcs.
In particular, the following theorem, which is the main result of \cite{K24},
shows that the normal form of a system of bridge arcs with respect to the trivial system of bridge arcs $\ep=\{\ep_1, \ep_2, \ep_3\}$
is unique.

\begin{Thm}[\cite{K24}]\label{unique normal form}
Let $\alpha=\{\alpha_1,\alpha_2,\alpha_3\}$ be a system of bridge arcs in normal form with respect to $\ep=\{\ep_1, \ep_2, \ep_3\}$
in the fixed trivial rational 3-tangle $T = (B,\ep)$. Then, up to isotopy, there is a unique normal form with respect to $\ep$, namely $\ep$ itself.
\end{Thm}

\begin{proof}
This is Theorem 1 in \cite{K24}.
\end{proof}

The following establishes that any systems of bridge arcs satisfying the rectangle condition must be in normal form.

\begin{Prop}\label{L4}
Let $(T_1, T_2;S)$ be a $3$-bridge decomposition of a knot $K$. Let $\alpha=\{\alpha_1,\alpha_2,\alpha_3\}$ and $\beta=\{\beta_1,\beta_2,\beta_3\}$ be systems of bridge arcs for $T_1$ and $T_2$ respectively that satisfy the rectangle condition. Then $\beta$ is in normal form with respect to $\alpha$ and vice versa.
\end{Prop}

\begin{proof}
For a contradiction, assume that $x$ and $y$ are two adjacent intersections in $\alpha_i$ that belong to the same $\beta_j$ for some $i, j\in\{1,2,3\}$. Let $\eta$ be the short subarc of $\alpha_i$ with endpoints $x$ and $y$, chosen so that it does not intersect $\be$ in its interior. Note that $\eta$ does not form a bigon with any subarc of $\beta_j$. If it did, we could isotope $\beta_j$ to reduce its intersection number with $\al$, which would contradict the assumption that the systems are in minimal general position.

Thus, $\eta$ together with a subarc of $\beta_j$ separates the other two bridge arcs different from $\beta_j$ among $\{\beta_1,\beta_2,\beta_3\}$. However, this is impossible, since the rectangle condition requires the existence of three types of rectangles between the two bridge arcs without intersecting $\eta$ or $\beta_j$. A similar argument can be applied when the roles of $\al$ and $\be$ are exchanged. This completes the proof.
\end{proof}

\section{Rectangle conditions and Strong irreducibility}\label{B3}

In this section, we define the rectangle condition and strong irreducibility for both a Heegaard splitting of a 3-manifold
and a bridge decomposition of a knot, and we show that the definitions on the two sides are equivalent; these constitute some of the main results of this paper.

The rectangle condition was originally defined in \cite{CG} for a genus $g$ Heegaard splitting $(V, W; \Sigma)$ of a
3-manifold $M$, where $\Sigma$ is a genus $g$ Heegaard surface that decomposes $M$ into two handlebodies $V$ and $W$.
Since we are primarily interested in a genus $2$ Heegaard splitting of a 3-manifold, we focus on describing the rectangle condition
for a genus $2$ Heegaard splitting $(V, W; \Sigma)$ of a 3-manifold.

Let $\mathcal{D}=\{D_1, D_2, D_3\}$ and $\mathcal{E}=\{E_1, E_2, E_3\}$ be collections of pairwise disjoint non-isotopic essential disks in $V$
and $W$ respectively. Then $\partial\mathcal{D}$ and $\partial\mathcal{E}$ induce pants decompositions $\mathcal{P}$ and $\mathcal{Q}$ of $\Sigma(=\partial V=\partial W)$,
each of which consists of two pants. We assume that $\partial\mathcal{P}$ and $\partial\mathcal{Q}$ intersect transversely and minimally,
and have three boundary components $\{a_1,b_1,c_1\}$ and $\{a_2,b_2,c_2\}$ respectively.
The two pants $P\in\mathcal{P}$ and $Q\in\mathcal{Q}$ are said to be \textit{tight} if, for each of the nine $4$-tuples of combinations listed below, there exists a rectangle $R$ embedded in $P$ and $Q$ such that the interior of $R$ is disjoint from $\partial P\cup \partial Q$, and the four edges of $\partial R$ are subarcs of the corresponding $4$-tuple:
\[
\begin{split}
&(a_1,b_1,a_2,b_2),\hskip 20pt (a_1,b_1,a_2,c_2),\hskip 20pt (a_1,b_1,b_2,c_2),\\
&(a_1,c_1,a_2,b_2),\hskip 20pt (a_1,c_1,a_2,c_2),\hskip 20pt (a_1,c_1,b_2,c_2),\\
&(b_1,c_1,a_2,b_2),\hskip 21pt (b_1,c_1,a_2,c_2),\hskip 21pt (b_1,c_1,b_2,c_2).
\end{split}
\]
\vskip 4pt

\begin{Def}[\textbf{Rectangle condition for a genus $2$ Heegaard splitting}] \label{Rectangle condition for Heegaard splitting}
Pants decompositions $\mathcal{P}$ and $\mathcal{Q}$ are said to satisfy a rectangle condition if every pair of a pant in $\mathcal{P}$ and a pant in $\mathcal{Q}$ is tight.
If there exist collections $\mathcal{D}=\{D_1, D_2, D_3\}$ and $\mathcal{E}=\{E_1, E_2, E_3\}$ of pairwise disjoint non-isotopic essential disks in $V$
and $W$ respectively that give pants decompositions $\mathcal{P}$ and $\mathcal{Q}$ satisfying the rectangle condition,
then the genus $2$ Heegaard splitting $(V, W; \Sigma)$ of the 3-manifold is said to satisfy a rectangle condition.
\end{Def}

Turning to a 3-bridge decomposition, analogously the first author \cite{K20} defined a rectangle condition for a $3$-bridge decomposition $(T_1, T_2;S)$ of a knot $K$ with $T_1=(B_1, \tau_1)$ and $T_2=(B_2, \tau_2)$.
Let $\mathcal{D}_i=\{D_{i1}, D_{i2}, D_{i3}\}$ be a system of compressions for $T_i$ $(i=1,2)$, and let $S_{\mathcal{D}_i}$ denote a connected component with
three circle boundaries in $\overline{S- \partial \mathcal{D}_i}$. In fact, $S_{\mathcal{D}_i}$ is a pair of pant. More generally, for a $n$-bridge decomposition, $S_{\mathcal{D}_i}$ admits a pants decomposition consisting of $(n-2)$ pairs of pants.

\begin{Def}
[\textbf{Rectangle condition for a $3$-bridge decomposition using compressions}] \label{for a 3-bridge decomposition using compressions}
Two pants $S_{\mathcal{D}_1}$ and $S_{\mathcal{D}_2}$ are said to satisfy the rectangle condition if $S_{\mathcal{D}_1}$ and $S_{\mathcal{D}_2}$ are tight.
If there exist systems of compressions $\mathcal{D}_1=\{D_{11}, D_{12}, D_{13}\}$ for $T_i$ and $\mathcal{D}_2=\{D_{21}, D_{22}, D_{23}\}$ for $T_2$
whose corresponding pants $S_{\mathcal{D}_1}$ and $S_{\mathcal{D}_1}$ satisfy the rectangle condition, then a $3$-bridge decomposition $(T_1,T_2;S)$ is said to satisfy a rectangle condition.
\end{Def}

By definition, the boundary of a surface $S_{\mathcal{D}_i}$ consists of the boundaries of the compressions $D_{i1}, D_{i2}, D_{i3}$.
Recall that, as illustrated in Figure~\ref{Rectangle-1}, the boundary of a regular neighborhood of a bridge disk $D_i$ in $B$ is the union
of two disks $E_i$ and $E'_i$ glued along their boundaries, where $E_i$ is a compression and $E'_i$ is a subset of $\partial B$ containing a bridge arc $\alpha_i$.
Therefore, the boundary of a compression can be identified with the boundary of the regular neighborhood of a bridge arc in $\partial B$.
Since the rectangle condition for a $3$-bridge decomposition in Definition~\ref{for a 3-bridge decomposition using compressions} is basically defined by a tight condition for the boundaries of compressions, it can equivalently be formulated using bridge arcs as follows.

\begin{Def}
[\textbf{Rectangle condition for a $3$-bridge decomposition using bridge arcs}] \label{for a 3-bridge decomposition using bridge arcs}
Let $\alpha=\{\alpha_{1}, \alpha_{2}, \alpha_{3}\}$ and $\beta=\{\beta_{1}, \beta_{2}, \beta_{3}\}$ be systems of bridge arcs for $T_1$ and $T_2$ respectively.
We say that the systems of bridge arcs $\alpha$ and $\beta$ satisfy the $\textit{rectangle condition}$ if, for each of the nine 4-tuples described below, there exists a rectangle $R$ in $S$ such that the interior of $R$ is disjoint from $\alpha\cup\beta$ and each of the four edges of $\partial R$ is a subarc of exactly one of the four entries of the 4-tuple:
\[
\begin{split}
&(\al_{1},\al_{2},\be_{1},\be_{2}),\hskip 20pt (\al_{1},\al_{2},\be_{1},\be_{3}),\hskip 20pt (\al_{1},\al_{2},\be_{2},\be_{3}),\\
&(\al_{1},\al_{3},\be_{1},\be_{2}),\hskip 20pt (\al_{1},\al_{3},\be_{1},\be_{3}),\hskip 20pt (\al_{1},\al_{3},\be_{2},\be_{3}),\\
&(\al_{2},\al_{3},\be_{1},\be_{2}),\hskip 20pt (\al_{2},\al_{3},\be_{1},\be_{3}),\hskip 20pt (\al_{2},\al_{3},\be_{2},\be_{3}).
\end{split}
\]

If there exist systems of bridge arcs for $T_1$ and $T_2$ that satisfy the rectangle condition, then a $3$-bridge decomposition $(T_1,T_2;S)$ is said to satisfy the rectangle condition.
\end{Def}

It is evident that the two definitions of the rectangle condition for a $3$-bridge decomposition in Definition~\ref{for a 3-bridge decomposition using compressions}
and in Definition~\ref{for a 3-bridge decomposition using bridge arcs} are equivalent.
Furthermore, by considering the double branched cover of a 3-bridge decomposition of a knot in $S^3$, which yields
a genus $2$ Heegaard splitting of a 3-manifold, we see that the rectangle condition for a genus $2$ Heegaard splitting in Definition~\ref{Rectangle condition for Heegaard splitting} corresponds naturally to the rectangle condition for a $3$-bridge decomposition in Definition~\ref{for a 3-bridge decomposition using compressions}.

In order to see the relationship, we first review some properties of essential disks in a genus 2 handlebody(or more generally in a genus $g$ handlebody).
Let $D$ and $E$ be essential disks in a genus 2 handlebody $H$ which intersect each other transversally and minimally.
Let $\delta$ be an outermost arc of intersection on $E$. Then $\delta$ cuts a disk $\Delta$ off from $E$ and divides $D$ into two subdisks $D'_1$ and $D'_2$.
Define $D_1=D'_1\cup_\delta \Delta$ and $D_2=D'_2\cup_\delta \Delta$. Then both are essential disks. By a slight isotopy, we may
assume that both $D_1$ and $D_2$ are disjoint from $D$. We say that $D_1$ and $D_2$ are obtained by \textit{disk-surgery} on $D$ along $E$(or along a subdisk of $E$).
See Figure~\ref{Rectangle-14}. Note that both $D_1$ and $D_2$ intersect $E$ fewer times than $D$ does. Also it is not hard to show that if $D$ is a nonseparating
essential disk, then the disks $D_1$ and $D_2$ obtained by disk-surgery on $D$ along $E$ are also nonseparating essential disks.
Similarly, we can define disk-surgery on a collection $\mathcal{D}$ along another collection $\mathcal{E}$ where $\mathcal{D}$ and $\mathcal{E}$ are sets of pairwise disjoint essential disks.

\begin{figure}[tbp]
\centering
\includegraphics[width = 0.75\textwidth]{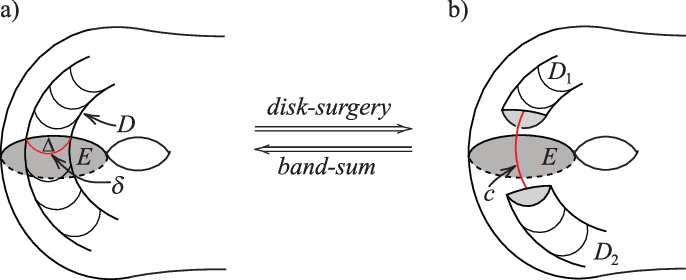}
\caption{Disk-surgery on $D$ along $E$ and band-sum $D$ of $D_1$ and $D_2$ along $c$.}
\label{Rectangle-14}
\end{figure}

%from hered

The disk-surgery operation decomposes somewhat one essential disk into two essential disks. Conversely, we can restore the decomposed essential disk from two resulting essential disks by performing the so-called ``band-sum" operation. Let $D_1$ and $D_2$ be disjoint non-isotopic essential disks
on a genus 2 handlebody $H$, and $c$ an arc on $\partial H$ with one endpoint on $\partial D_1$ and the other on $\partial D_2$,
such that the interior of $c$ is disjoint from $D_1\cup D_2$. The regular neighborhood of $c$ in $\partial H$ forms a band $B$ connecting
$D_1$ and $D_2$.
Then $D_1\cup_\partial B\cup_\partial D_2$ is a disk which after a slight isotopy gives a properly embedded disk $D$ in $H$.
We say that $D$ is a \textit{band-sum} of $D_1$ and $D_2$ along an arc $c$. Note that since $D_1$ and $D_2$ are disjoint non-isotopic essential disks, a band-sum $D$ is also
essential. Thus we see that if $D_1$ and $D_2$ arise from disk-surgery on $D$ along a subdisk $\Delta$ of $E$, then conversely $D$ can be obtained as a band-sum
of $D_1$ and $D_2$ along an arc $c$ that intersects $\Delta$ transversally in a single point. See Figure~\ref{Rectangle-14}.

\begin{Lem}\label{nonseparating disks by bandsum}
Let $\mathcal{E}=\{E_1, E_2, E_3\}$ be a set of disjoint nonseparating essential disks in a genus 2 handlebody as illustrated in Figure~\ref{Rectangle-15}.
Suppose $D$ is a nonseparating disk in the handlebody. Then there exists a sequence of $\mathcal{D}_1=\{D\}, \mathcal{D}_2, \ldots, \mathcal{D}_n$
such that for each $i=2, 3, \ldots, n$, $\mathcal{D}_i$ is a set of disjoint nonseparating essential disks obtained by disk-surgery on $\mathcal{D}_{i-1}$
along $\mathcal{E}$, and $\mathcal{D}_n$ consists entirely of disjoint copies of the disks in $\mathcal{E}$.
\end{Lem}
%here.

\begin{proof}
If $D$ is disjoint from $\mathcal{E}$, then $D$ must be isotopic to one of $E_1, E_2, E_3$, in which case we are done. Suppose that $D$ intersects $\mathcal{E}$ transversally and minimally. Then take one outermost disk and perform disk-surgery on $D$ along $\mathcal{E}$ to make two nonseparating essential disks $D_1$ and $D_2$. Let $\mathcal{D}_2=\{D_1, D_2\}$. As noted, $D_1$ and $D_2$ are nonseparating essential disks and $\mathcal{D}_2$ intersects $\mathcal{E}$ fewer than $D$ does. Now applying the mathematical induction on the number of intersections of $D$ with $\mathcal{E}$, we eventually obtain the result as desired.
\end{proof}

\begin{figure}[tbp]
\centering
\includegraphics[width = 0.65\textwidth]{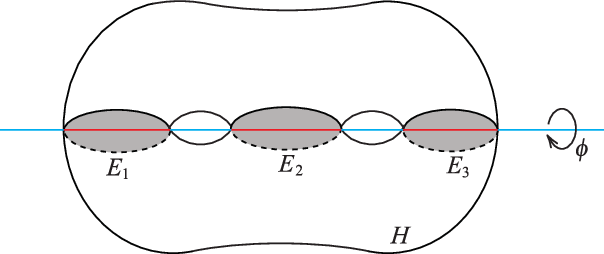}
\caption{Nonseparating essential disks in a genus 2 handlebody $H$ and a hyperelliptic involution $\phi$ on $H$.}
\label{Rectangle-15}
\end{figure}

\begin{Lem}\label{nonseparating disks to compressions}
Let $H$ be a genus 2 handlebody with $\mathcal{E}=\{E_1, E_2, E_3\}$ a set of disjoint nonseparating essential disks in $H$ and $T=(B, \tau)$ a rational 3-tangle obtained from $H$ by
taking a hyperelliptic involution $\phi$ on $H$ as shown in Figure~\ref{Rectangle-15}.
Then every nonseparating essential disk of $H$ can be isotoped to be disjoint from Fix($\phi$) so that under $\phi$ it is carried to a compression disk in $T$.
\end{Lem}

\begin{proof}
Let $D$ be a nonseparating essential disk in $H$. Then the intersection of $D$ and $\mathcal{E}$ consists of arcs in $\mathcal{E}$.
Since Fix($\phi$) is the union of the three red arcs in $E_1\cup E_2\cup E_3$ as shown in Figure~\ref{Rectangle-15}, we can isotope $D\cap\mathcal{E}$ to be disjoint
from Fix($\phi$). Consequently $D$ itself can be isotoped to be disjoint from Fix($\phi$).

Now we show that the isotoped disk $D$ is sent to a compression in $T$ by $\phi$.
By Lemma~\ref{nonseparating disks by bandsum}, there exists a sequence of collections of disks $\mathcal{D}_1=\{D\}, \mathcal{D}_2, \ldots, \mathcal{D}_n$ such that for each $i=2, 3, \ldots, n$, $\mathcal{D}_i$ consists of pairwise disjoint nonseparating essential disks obtained by disk-surgery on $\mathcal{D}_{i-1}$
along $\mathcal{E}$, and moreover $\mathcal{D}_n$ consists precisely of disjoint copies of the disks in $\mathcal{E}$. Since the disk-surgery operation forward is resumed by the operation of band-sum backward, we may start from $\mathcal{D}_n$ and successively apply band-sum operations to recover $\mathcal{D}_{n-1}, \mathcal{D}_{n-2}, \ldots, \mathcal{D}_1=\{D\}$.

Obviously, the copies of essential disks $E_1, E_2, E_3$ in Figure~\ref{Rectangle-15}, which are disjoint from Fix($\phi$), are mapped to compressions in $T=(B, \tau)$ by the hyperelliptic involution $\phi$. Moreover, the band-sum of any two disks among $E_1, E_2, E_3$ is also sent to a compression in $T=(B, \tau)$. Therefore, every essential disk in $\mathcal{D}_{n-1}$ descends to a compression in $T=(B, \tau)$.

We now proceed by mathematical induction on $n$. Suppose $D^i$ in $\mathcal{D}_i$ is obtained by a band-sum of two essential disks $D^{i+1}_1$ and $D^{i+1}_2$ in $\mathcal{D}_{i+1}$. In other words, there exists an arc $c$ in $\partial H$ connecting the disks $D^{i+1}_1$ and $D^{i+1}_2$ to make $D^i$ by band-sum operation along $c$. By the induction, $D^{i+1}_1$ and $D^{i+1}_2$ are mapped to compressions $\check{D}^{i+1}_1$ and $\check{D}^{i+1}_2$ in $T$. Furthermore for the arc $c$, there exists an arc $c'$ isotopic to $c$ (with its endpoints fixed) whose image under $\phi$ is an embedded arc $\check{c'}$ in $T$ connecting the compressions $\check{D}^{i+1}_1$ and $\check{D}^{i+1}_2$. This implies that the image of $D^i$ under $\phi$ is precisely the band-sum of $\check{D}^{i+1}_1$ and $\check{D}^{i+1}_2$ along the arc $\check{c'}$, which is again a compression in $T$. This completes the proof.
\end{proof}

\begin{Lem}\label{separating disk with rectangle condition}
Let $(V, W; \Sigma)$ be a genus 2 Heegaard splitting of a 3-manifold $M$. Suppose $\mathcal{P}$ and $\mathcal{Q}$ are pants decompositions of $\Sigma$ from $V$ and $W$ respectively, that satisfy the rectangle condition. Then there exist pants
decompositions $\mathcal{P}'$ and $\mathcal{Q}'$, induced from the sets of pairwise disjoint non-isotopic nonseparating essential disks in $V$ and $W$ respectively, which satisfy the rectangle condition.
\end{Lem}

\begin{figure}[tbp]
\centering
\includegraphics[width = 1\textwidth]{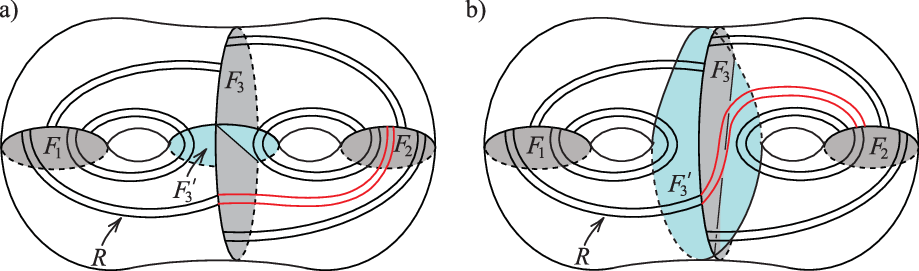}
\caption{The rectangle types of the two pants decompositions in $\Sigma$ coming from $\partial F_1, \partial F_2, \partial F_3$ and $k\circ h(\partial E_1), k\circ h(\partial E_2), k\circ h(\partial E_3)$.}
\label{Rectangle-16}
\end{figure}

\begin{proof}
Let $\mathcal{D}=\{D_1, D_2, D_3\}$ and $\mathcal{E}=\{E_1, E_2, E_3\}$ be the collections of pairwise disjoint non-isotopic essential disks in $V$ and $W$ respectively that induce the pants decompositions $\mathcal{P}$ and $\mathcal{Q}$ of $\Sigma$, which satisfy the rectangle condition.
If both $\mathcal{D}=\{D_1, D_2, D_3\}$ and $\mathcal{E}=\{E_1, E_2, E_3\}$ consist of all nonseparating essential disks, then we are done.

Now we assume that $D_1$ and $D_2$ are nonseparating, while $D_3$ is separating. Note that $\mathcal{D}$ must contain at most one separating disk.
Our goal is to find a nonseparating essential disk
$D_3'$ such that $\mathcal{D}'=\{D_1, D_2, D_3'\}$ and $\mathcal{E}=\{E_1, E_2, E_3\}$ satisfy the rectangle condition.

To achieve this, we perform suitable homeomorphisms on $V$ so that $\mathcal{D}=\{D_1, D_2, D_3\}$ and the pants $\mathcal{P}$ and $\mathcal{Q}$ in $V$
are arranged in a canonical form. In this setting, we locate a nonseparating essential disk
$F'_3$, and then apply the inverses of the chosen homeomorphisms on $V$ to obtain the desired nonseparating essential disk $D_3'$.
Since homeomorphisms on $V$ preserve the rectangle condition, the resulting
collections $\mathcal{D}'$ and $\mathcal{E}$ satisfy the required property.

First we apply a homeomorphism $h$ from $V$ to itself
sending $D_1, D_2, D_3$ to essential disks $F_1, F_2, F_3$ respectively as described in Figure~\ref{Rectangle-16}. Then the two pants decompositions in $\Sigma$ arising from $\partial F_1, \partial F_2, \partial F_3$ and $h(\partial E_1), h(\partial E_2), h(\partial E_3)$ satisfy the rectangle condition. Second, by performing a homeomorphism $k$ on $V$, consisting of a sequence of Dehn twists along the disks $F_1, F_2$, or $F_3$, if necessary, we can assume that the rectangle types in the rectangle condition for the two pants decompositions in $\Sigma$, obtained from $\partial F_1, \partial F_2, \partial F_3$ and $k\circ h(\partial E_1), k\circ h(\partial E_2), k\circ h(\partial E_3)$, appear as shown in Figure~\ref{Rectangle-16}.

Now we focus on the extension of the rectangle type $R$ emanating out from the disk $F_1$ and terminating at the disk $F_3$ as shown in Figure~\ref{Rectangle-16}. Since all rectangle types are disjoint, there are only two possible extensions of $R$ to the other once-punctured tori of $\partial H$ ending at the disk $F_2$, which are indicated by the red lines in Figures~\ref{Rectangle-16}a and \ref{Rectangle-16}b. For each extension, we consider a nonseparating essential disk $F'_3$ as shown in Figure~\ref{Rectangle-16}a or b, and replace $F_3$ with $F'_3$. We can observe from Figure~\ref{Rectangle-16} that there exist rectangle types between $F_1$ and $F'_3$, and between $F_2$ and $F'_3$. Also due to the extension of $R$ and the choice of the disk $F'_3$, there is a rectangle type between $F_1$ and $F_2$. This implies that the two pants decompositions in $\Sigma$ arising from $\partial F_1, \partial F_2, \partial F'_3$ and $k\circ h(\partial E_1), k\circ h(\partial E_2), k\circ h(\partial E_3)$ satisfy the rectangle condition.

If $\mathcal{E}=\{E_1, E_2, E_3\}$ contains a separating essential disk, we apply the same procedure above once again to complete the proof of this lemma.
\end{proof}

%here

\begin{Thm}\label{main result-intro-1}
Let $(T_1, T_2; S)$ be a 3-bridge decomposition of a knot $K$ in $S^3$ and
$(V, W;\Sigma)$ the genus 2 Heegaard splitting obtained by the double branched covering of
$(T_1, T_2; S)$. Then $(V, W;\Sigma)$ satisfies the rectangle condition if and only if $(T_1, T_2; S)$ satisfies the rectangle condition.
\end{Thm}

\begin{proof}
First, we prove the `if' part, which is relatively straightforward. By the definition of the double branched covering, a compression in a rational 3-tangle lifts to two copies of an (nonseparating) essential disk in a genus two handlebody. This implies that pants decompositions in a bridge sphere $S$ induce pants decompositions in the Heegaard surface $\Sigma$.
Moreover, each rectangle type in the rectangle condition for the pants decompositions in $S$ lifts to two copies of the corresponding rectangle in $\Sigma$.
Therefore, the pants decompositions in the Heegaard surface $\Sigma$ induced from those in a bridge sphere $S$ satisfy the rectangle condition, establishing the `if' part of the theorem.

Now, we prove the `only if' part. Applying Lemma~\ref{separating disk with rectangle condition}, we may assume that there exist
collections of pairwise disjoint non-isotopic nonseparating essential disks $\mathcal{D}=\{D_1, D_2, D_3\}$ in $V$ and $\mathcal{E}=\{E_1, E_2, E_3\}$ in $W$
which give rise to pants decompositions $\mathcal{P}$ and $\mathcal{Q}$ of $\Sigma$ satisfying the rectangle condition.
By Lemma~\ref{nonseparating disks to compressions}, the nonseparating essential disks $D_i$ and $E_i$ for $i=1,2,3$ are mapped to compressions in the
rational 3-tangles $T_1$ and $T_2$ respectively under hyperelliptic involutions. Consequently, the pants decompositions $\mathcal{P}$ and $\mathcal{Q}$ in $\Sigma$
induce the pants decompositions $\mathcal{\check{P}}$ and $\mathcal{\check{Q}}$ in the bridge sphere $S$. Furthermore, under the hyperelliptic involutions,
each rectangle type in $\Sigma$ descends to a rectangle type in $S$. Therefore, the two pants decompositions $\mathcal{\check{P}}$ and $\mathcal{\check{Q}}$ satisfy the rectangle condition, and hence $(T_1, T_2; S)$ satisfies the rectangle condition.
\end{proof}

Now we relate two notions of the strong irreducibility for a genus $2$ Heegaard splitting of a 3-manifold $M$ and for a 3-bridge decomposition of a knot in $S^3$.
Recall that a Heegaard splitting $(V, W; \Sigma)$ of a 3-manifold $M$ is strongly irreducible if every pair consisting of an essential disk in $V$ and an essential disk in $W$ intersects. Similarly, an $n$-bridge decomposition $(T_1, T_2;S)$ of a knot $K$ in $S^3$ is strongly irreducible if every pair of a compression in $T_1$ and a compression in $T_2$ intersects. As with the rectangle condition, the notions of strong irreducibility on both sides are equivalent.

\begin{Thm}\label{main result-intro-2}
Let $(T_1, T_2; S)$ be a 3-bridge decomposition of a knot $K$ in $S^3$ and
$(V, W;\Sigma)$ the genus 2 Heegaard splitting obtained by the double branched covering of
$(T_1, T_2; S)$. Then $(V, W;\Sigma)$ is strongly irreducible if and only if $(T_1, T_2; S)$
is strongly irreducible.
\end{Thm}

\begin{proof}
First, we prove the `only if' part, which is easier to prove. Suppose for contradiction that $(T_1, T_2; S)$
is not strongly irreducible. Then there exists a compression $D$ in $T_1$ and a compression $E$ in $T_2$ which
are disjoint. By the property of the double branched covering, each compression of $D$ and $E$ lifts to two copies of a nonseparating essential disk in $V$ and $W$, respectively.
Therefore, the corresponding essential disks in $V$ and $W$ are disjoint, contradicting the assumption that $(V, W; \Sigma)$ is strongly irreducible. Hence, $(T_1, T_2; S)$ must be strongly irreducible.

Next, we prove the `if' part. Suppose for contradiction that $(V, W;\Sigma)$ is not strongly irreducible, i.e.,
it is weakly reducible. It is well known that a genus 2 weakly
reducible Heegaard splitting is a reducible Heegaard splitting.
Therefore, $(V, W;\Sigma)$ is one of
the following:

\begin{itemize}
\item[(i)] the connected sum of two genus 1 Heegaard splittings of $S^3$,
\item[(ii)] the connected sum of a genus 1 Heegaard splitting of a lens space and a genus 1
Heegaard splitting of $S^3$,
\item[(iii)] the connected sum of two genus 1 Heegaard splittings of lens spaces.
\end{itemize}

For (i), (ii), (iii),  the corresponding $3$-bridge decomposition $(T_1, T_2; S)$ would be a $3$-bridge decomposition of an unknot, a 2-bridge
knot, a connected sum of two 2-bridge knots, respectively. For all such knots, a 3-bridge
decomposition is always weakly reducible, a contradiction.
\end{proof}

\begin{figure}[tbp]
\centering
\includegraphics[width = 0.9\textwidth]{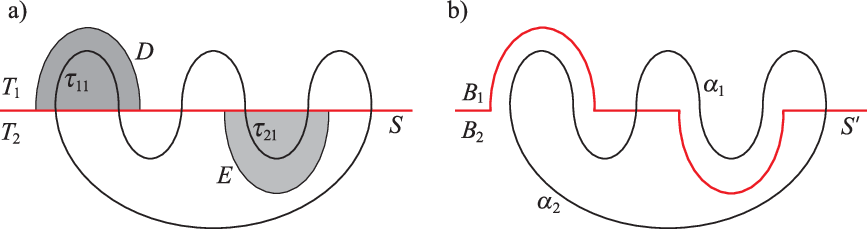}
\caption{A 3-bridge decomposition $(T_1, T_2;S)$ of a prime knot $K$ in $S^3$ that is not strongly irreducible.}
\label{Rectangle-2}
\end{figure}

Lastly, we will show that for a 3-bridge knot $K$, the strong irreducibility of a 3-bridge decomposition of $K$
is equivalent to the primeness of $K$.

%here

\begin{Thm}\label{main result-intro3}
A 3-bridge decomposition of a 3-bridge knot $K$ in $S^3$ is strongly irreducible if and only if
$K$ is prime.
\end{Thm}

\begin{proof}
Let $(T_1, T_2;S)$ be a 3-bridge decomposition of a 3-bridge knot $K$, where
the bridge sphere $S$ decomposes $(S^3, K)$ into two rational tangles $T_1=(B_1, \tau_1)$ and
$T_2=(B_2, \tau_2)$.

We first prove the `if' part. Suppose $(T_1, T_2;S)$ is not strongly irreducible.
Then $(T_1, T_2;S)$ is weakly reducible. Therefore, there exist compressions $D$ in $T_1$ and $E$ in $T_2$ which are disjoint
each other.
Figure~\ref{Rectangle-2}a shows a schematic picture for the compressions $D$ and $E$ in $(T_1, T_2;S)$.
We can observe that there are tangle arcs $\tau_{11}$ and $\tau_{21}$ of $\tau_1$ and $\tau_2$ respectively which are
separated by $D$ and $E$ from the other two tangle arcs of $\tau_1$ and $\tau_2$ respectively.
Now we perform the weak reduction so that, as shown in Figure~\ref{Rectangle-2}b, we get a 2-sphere $S'$ decomposing $(S^3, K)$ into
$(B_1, \alpha_1)$ and $(B_2, \alpha_2)$, where $B_i$ is a 3-ball and $\alpha_i$ is an arc in $B_i$ for $i=1,2$.
For each $i (i=1,2)$ we attach $(B_0, \alpha_0)$ of a 3-ball $B_0$ and a trivial arc $\alpha_0$ to $(B_i, \alpha_i)$.
Then we get a knot $K_i=\alpha_i\cup \alpha_0$ in $S^3$. We call $K_i$ the knot associated to $\alpha_i$.

Now there are three cases for $\alpha_1$ and $\alpha_2$: (1) both are trivial arcs, (2) only one of $\alpha_1$ and $\alpha_2$ is trivial, (3) both are knotted arcs.

Case (1): Suppose both $\alpha_1$ and $\alpha_2$ are trivial arcs. Then $K$ is the unknot, a contradiction.

Case (2): Suppose only one of $\alpha_1$ and $\alpha_2$, say $\alpha_1$, is trivial. Then $K$ admits a 2-bridge
decomposition implying that $K$ is a 2-bridge knot, a contradiction.

Case (3): Suppose both $\alpha_1$ and $\alpha_2$ are knotted arcs. Then $K_1$ and $K_2$ are knotted, and $K=K_1 \# K_2$,
implying that $K$ is a composite knot, a contradiction.

Now we prove the `only if' part. Suppose for a contradiction that $K$ is a composite knot. It is known, for example
in \cite{Sb54}, \cite{D92}, \cite{HS01}, \cite{Sl03}, that
for any bridge decomposition of a composite knot, there exists a decomposing sphere $F$ of
the connected sum such that $F$ intersects the bridge sphere in a single circle, and hence the
bridge decomposition can be decomposed into two bridge decompositions of the summands.
Thus, in our case, $F$ decomposes $(T_1, T_2; S)$ into two bridge decompositions of
2-bridge knots, and therefore it is weakly reducible, a contradiction.
\end{proof}

\section{Strongly irreducible 3-bridge decomposition with no rectangle condition }\label{B4}

In this section, we prove Theorem~\ref{main result}, which is the main result of this paper.
To prove Theorem~\ref{main result}, we consider the knot $8_5$ in the Rolfsen knot table of \cite{R76},
which is known to be a prime 3-bridge knot.
Figure~\ref{Rectangle-7}a shows a diagram of $8_5$ and also a 3-bridge decomposition $(T_1, T_{\epsilon};S)$ of the knot $8_5$.
Note that the diagram of $8_5$ in Figure~\ref{Rectangle-7}a is exactly the same as that in the Rolfsen knot table.
Figure~\ref{Rectangle-7}b illustrates a system $\del=\{\del_1, \del_2, \del_3\}$ of the red, green, blue bridge arcs for $T_1$
and a system $\ep=\{\ep_1, \ep_2, \ep_3\}$ of the trivial bridge arcs for $T_\ep$.
Also Figure~\ref{Rectangle-21} shows the double branched cover of the 3-bridge decomposition $(T_1, T_{\epsilon};S)$ of the knot $8_5$ in Figure~\ref{Rectangle-7}b, where
the genus 2 surface $\tilde{S}$, and the simple closed curves $\tilde{\ep_1}, \tilde{\ep_2}, \tilde{\ep_3}$ and $\tilde{\del_1}, \tilde{\del_2}, \tilde{\del_3}$ are double branched covers of
the bridge sphere $S$, and the bridge arcs $\ep_1, \ep_2, \ep_3$ and $\del_1, \del_2, \del_3$ respectively.

\begin{figure}[tbp]
\centering
\includegraphics[width = 0.8\textwidth]{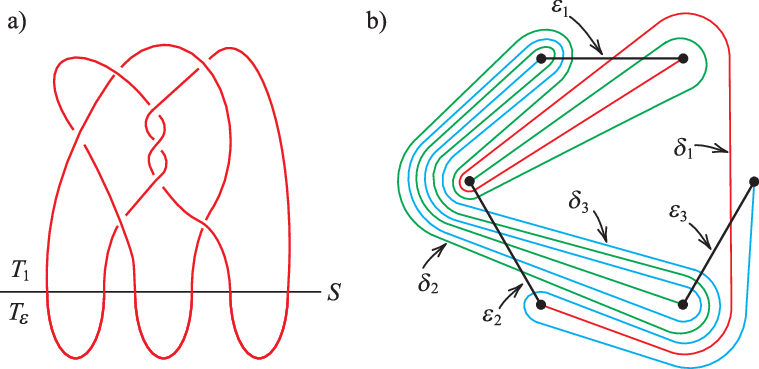}
\caption{(a) The 3-bridge decomposition $(T_1, T_{\epsilon};S)$ of the Knot $8_5$ in the knot table and (b) the bridge arc systems $\del=\{\del_1, \del_2, \del_3\}$ and $\ep=\{\ep_1, \ep_2, \ep_3\}$ for $T_1$ and $T_\ep$ respectively.}
\label{Rectangle-7}
\end{figure}

\begin{figure}[tbp]
\centering
\includegraphics[width = 0.8\textwidth]{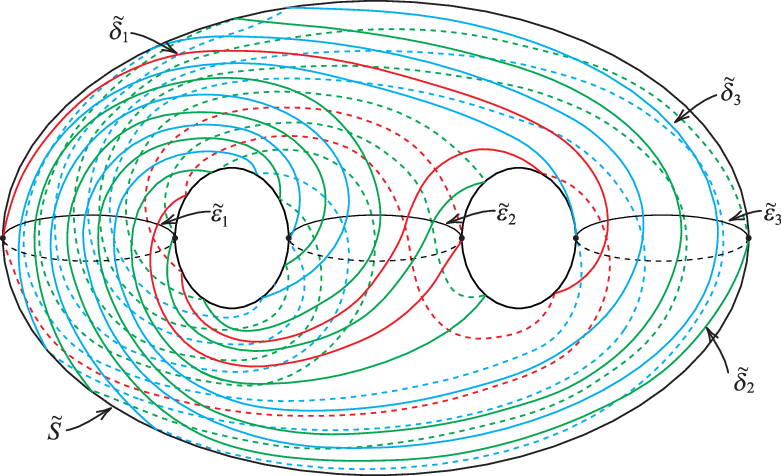}
\caption{Double branched cover of the 3-bridge decomposition $(T_1, T_{\epsilon};S)$ of the Knot $8_5$ in Figure~\ref{Rectangle-7}b.}
\label{Rectangle-21}
\end{figure}

We will show that the 3-bridge decomposition of $8_5$ in Figure~\ref{Rectangle-7}a is strongly irreducible but does not satisfy the rectangle condition,
To do this, we present the following theorems, which are the main results of this paper and prove Theorems~\ref{main result} and ~\ref{main result2}.

\begin{Thm}\label{main result3}
The $3$-bridge decomposition of the knot $8_5$ in Figure~\ref{Rectangle-7}a is strongly irreducible but does not satisfy the rectangle condition.
\end{Thm}

\noindent By Theorems~\ref{main result-intro-1} and ~\ref{main result-intro-2}, we obtain the following from Theorem~\ref{main result3}

\begin{Thm}\label{main result4}
The genus 2 Heegaard splitting obtained by the double branched covering of the $3$-bridge decomposition of the knot $8_5$ in Figure~\ref{Rectangle-7}a is strongly irreducible but does not satisfy the rectangle condition.
\end{Thm}

For the strong irreducibility of the $3$-bridge decomposition of the knot $8_5$ in Theorem~\ref{main result3}, it follows
immediately from Theorem~\ref{main result-intro3}. Regarding the proof of Theorem~\ref{main result3} for rectangle condition, we need some theory.
For the systems $\del=\{\del_1, \del_2, \del_3\}$ and $\ep=\{\ep_1, \ep_2, \ep_2\}$ of bridge arcs for $T_1$ and $T_\ep$ respectively in Figure~\ref{Rectangle-7}b
it follows that $\del_2$ and $\del_3$ do not have subarcs connecting the bridge arcs $\ep_1$ and $\ep_3$
without intersecting $\ep_2$. Therefore $\del$ and $\ep$ do not satisfy the rectangle condition.
By Lemma~\ref{infinitely many systems of bridge arcs}, there are infinitely many systems of bridge arcs in each of the tangles $T_1$ and $T_\ep$.
Thus in order to show that any pair of systems of bridge arcs for $T_1$ and $T_\ep$ do not satisfy the rectangle condition, we need some lemmas.

First, suppose $\beta=\{\beta_1,\beta_2,\beta_3\}$ is a system of bridge arcs for $T_{\epsilon}$ that is not isotopic to $\ep=\{\ep_1,\ep_2,\ep_3\}$.
It follows from Theorem~\ref{unique normal form} that $\beta$ is not in a normal form with respect to $\ep$.
Therefore, there exist two adjacent intersection points $p$ and $q$ of $\beta_{j_0}$ in $\ep_i$ for some $i, j_0\in\{1, 2, 3\}$.
The following lemma restricts subarcs of $\beta_{j_0}$ whose endpoints are $p$ or $q$.

\begin{figure}[tbp]
\centering
\includegraphics[width = 0.75\textwidth]{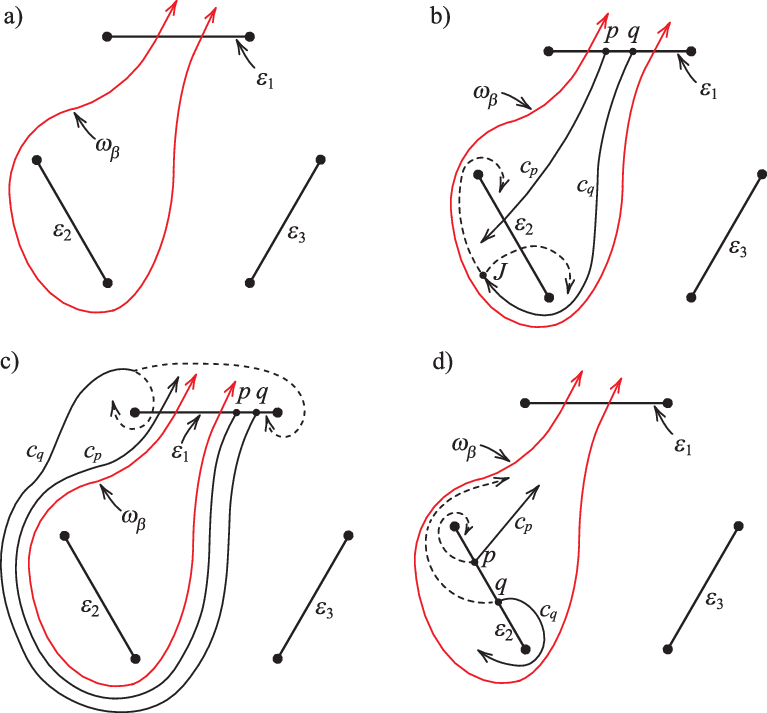}
\caption{A wave $\omega_\be$ of $\be_1$ which is based at $\ep_1$, and the two curves $c_p$ and $c_q$ starting from the adjacent points $p$ and $q$ which lie in $\ep_1$
and $\ep_2$.}
\label{Rectangle-19}
\end{figure}

\begin{Lem}\label{adjacent intersections}
Suppose that $\beta=\{\beta_1,\beta_2,\beta_3\}$ is a system of bridge arcs for $T_{\epsilon}$ that is not isotopic to $\ep=\{\ep_1,\ep_2,\ep_3\}$.
Let $\{p, q\}$ be a pair of adjacent intersection points of $\beta$ in $\ep_{i}$ which belong to the same bridge arc $\be_{j_0}$ for some $i, j_0\in\{1, 2, 3\}$.
Then $\{p, q\}$ satisfies one of the following conditions:\\
(i) $p$ and $q$ are the endpoints of parallel subarcs of $\be_{j_0}$ connecting $\ep_{i}$ and $\ep_j$ for $i, j\in\{1, 2, 3\}$.\\
(ii) $p$ and $q$ are the endpoints of a wave based at $\ep_{i}$, which is a subarc of $\be_{j_0}$.\\
(iii) One of $p$ and $q$ is an endpoint of a wave $\be_{j_0}$ based at $\ep_{i}$ and the other is an endpoint of a subarc of $\be_{j_0}$ connecting $\ep_{i}$ and $\ep_j$ for distinct $i, j\in\{1, 2, 3\}$
\end{Lem}

%here
\begin{proof}
Suppose for contradiction that the adjacent intersection points $p$ and $q$ in $\ep_{i}$ for some $i\in\{1, 2, 3\}$ don't satisfy any of the three conditions in the lemma. Let $c_p$ and $c_q$ be subarcs of $\be_{j_0}$ starting from $p$ and $q$ respectively on the same side of $\ep_i$.
Since $\be$ is not isotopic to $\ep$, Lemma~\ref{wave existence} implies that $\be$ has a wave $\omega_{\be}$ with respect to $\ep$. Without loss of generality, we may assume that the wave $\omega_\be$ belongs to a bridge arc $\be_1$ and is based at $\ep_1$ as shown in Figure~\ref{Rectangle-19}a.

First, assume that $\ep_i=\ep_1$, i.e., $p$ and $q$ lie on $\ep_1$. Without loss of generality, we may assume that there are two cases to consider depending on whether or not $p$ and $q$ lie inside the wave $\omega_\be$ as shown in Figures~\ref{Rectangle-19}b and ~\ref{Rectangle-19}c.
In the case where $p$ and $q$ lie inside the wave $\omega_\be$ in Figure~\ref{Rectangle-19}b, because of the assumption that $p$ and $q$ don't satisfy any of the three conditions in the lemma, after, if necessary, taking a multiple of half Dehn twists along the disk $E_2$, which is a regular neighborhood of the bridge arc $\ep_2$ in the boundary of the tangle, we may assume that $c_p$ and $c_q$ appear as shown in Figure~\ref{Rectangle-19}b where $c_p$ hits the bridge arc $\ep_2$ and $c_q$ turns around $\ep_2$. Then the curve $c_q$ has two possible directions to go at the point $J$, indicated by the dotted lines in Figure~\ref{Rectangle-19}b. Since it is assumed that $p$ and $q$ are adjacent and don't satisfy any of the three conditions in the lemma, the both possible directions of $c_q$ give rise to an infinite spiral, a contradiction.

For the case where $p$ and $q$ lie outside the wave $\omega_\be$ in Figure~\ref{Rectangle-19}c, similarly by the assumption for the adjacent points $p$ and $q$, we may assume that the curves $c_p$ and $c_q$ appear as shown in Figure~\ref{Rectangle-19}c, where $c_p$ runs along the wave $\omega_\be$ and hits $\ep_1$, while $c_q$ goes along the wave $\omega_\be$ and turns around $\ep_1$. Similarly as before, there are two possible directions to go for the extension of $c_q$, each of which either produces an infinite spiral or passes between the adjacent points $p$ and $q$, again a contradiction. Note that if either $c_p$ or $c_q$ hits $\ep_3$ first before meeting $\ep_1$ in this case, then it is easy to see that this leads to an infinite spiral or violates the assumption that $p$ and $q$ are adjacent and do not satisfy any of the three conditions in the lemma.

\begin{figure}[tbp]
\centering
\includegraphics[width = 0.75\textwidth]{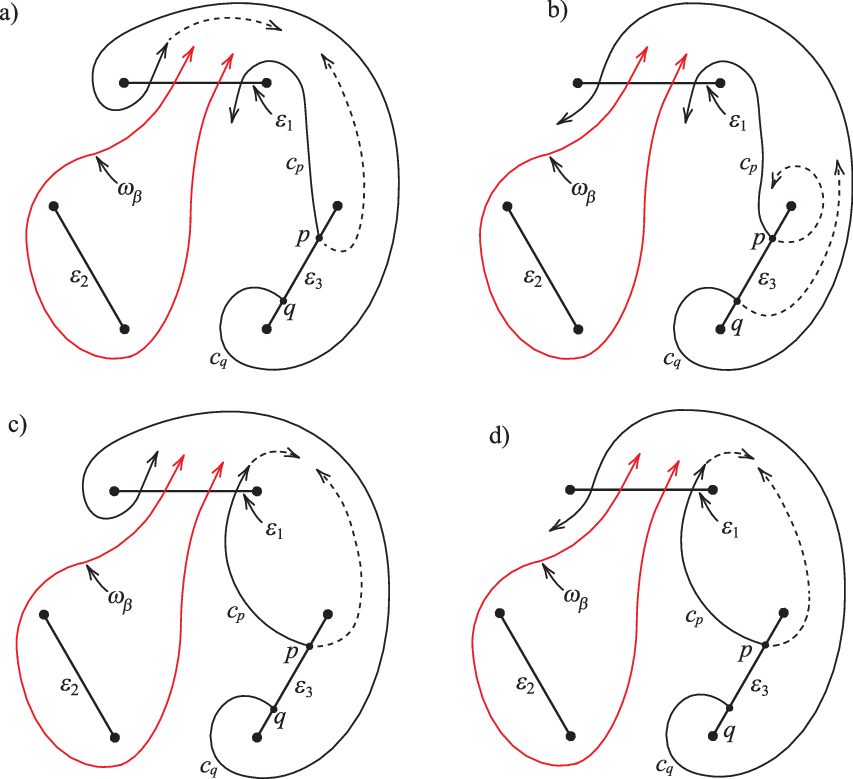}
\caption{A wave $\omega_\be$ of $\be_1$ which is based at $\ep_1$, and the two curves $c_p$ and $c_q$ starting from the adjacent points $p$ and $q$ which lie in $\ep_3$.}
\label{Rectangle-20}
\end{figure}

Second, we assume that $\ep_i=\ep_2$, i.e., $p$ and $q$ lie on $\ep_2$ as shown in Figure~\ref{Rectangle-19}d. Similarly to the first case, we may assume that the subarcs $c_p$ and $c_q$ have the configuration depicted in Figure~\ref{Rectangle-19}d. We then extend the curves $c_p$ and $c_q$ from the points $p$ and $q$ respectively to the opposite side of $\ep_2$ as indicated by the dotted lines in Figure~\ref{Rectangle-19}d. By the assumption that $p$ and $q$ are adjacent and do not satisfy any of the three conditions in the lemma, the dotted line of $c_p$ produces an infinite spiral, a contradiction.

Lastly, we assume that $\ep_i=\ep_3$, i.e., $p$ and $q$ lie on $\ep_3$. Because of the assumption for $p$ and $q$, there are only four cases for $c_p$ and $c_q$ to consider as illustrated in Figure~\ref{Rectangle-20}. By the same reason as before that $p$ and $q$ are adjacent and are assumed not to satisfy any of the three conditions in the lemma, infinite spirals occur from dotted lines for all the cases, a contradiction.
\end{proof}

Now we introduce another lemma used in the proof of Theorem~\ref{main result3}. For this lemma, we need some terminology analogous to
the concept of a wave. As usual, let $\alpha=\{\alpha_1,\alpha_2, \alpha_3\}$ be a system of bridge arcs of a rational $3$-tangle $T=(B, \tau)$,
where $\tau$ is the union of three strings $\tau_1, \tau_2, \tau_3$.
An \textit{essential arc} \textit{with respect to a system of bridge
arcs} $\al$ is defined to be an embedded arc $\eta$ in $\partial B$ such that its interior is disjoint from $\al$, both endpoints
of $\eta$ lies in the same bridge arc $\al_i$ for some $i\in\{1, 2, 3\}$, and $\eta$ is essential in $\partial B -N(\al)$.
We say that an \textit{essential arc} $\eta$ \textit{is based at} $\al_i$.
Note that it follows immediately from the definition of a wave that a wave with respect to $\al$ is an essential arc with respect to $\al$.
Also, as the property of a wave, if $\eta$ is an essential arc based at, say $\al_1$, then $\eta$, together with a subarc of $\alpha_1$, separates the
other two bridge arcs $\alpha_2$ and $\alpha_3$. See Figure~\ref{Rectangle-3}b.

\begin{figure}[tbp]
\centering
\includegraphics[width = 0.75\textwidth]{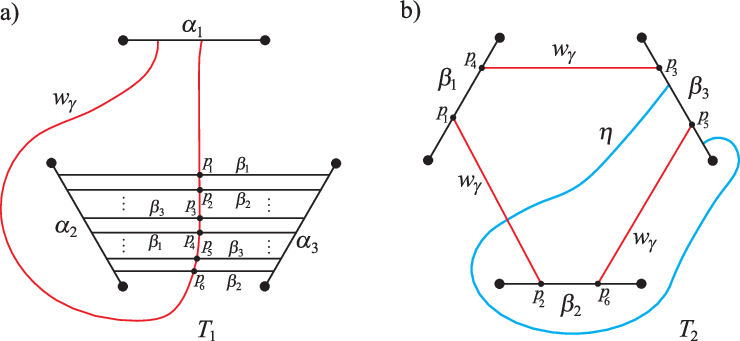}
\caption{Subarcs of $\gamma$ connecting $\beta_i$ and $\beta_j$ and intersecting an essential arc based at $\be_i$ for $i, j\in\{1, 2, 3\}$.}
\label{Rectangle-3}
\end{figure}

\begin{Lem}\label{any bridge arc connects}
Let $(T_1, T_2;S)$ be a $3$-bridge decomposition of a knot $K$. Suppose $\alpha=\{\alpha_1,\alpha_2,\alpha_3\}$ and $\beta=\{\beta_1,\beta_2,\beta_3\}$ are
systems of bridge arcs for $T_1$ and $T_{2}$ respectively which satisfy the rectangle condition. If $\gamma$ is a bridge arc for $T_1$, then
for every pair $\{i, j\}(i\neq j)$ in $\{1, 2, 3\}$, there exists a subarc of $\gamma$ connecting $\beta_i$ and $\beta_j$ whose interior is disjoint from $\be$.
Furthermore, for each $i\in\{1,2, 3\}$ every essential arc based at $\be_i$ intersects $\gamma$.
\end{Lem}

%here
\begin{proof}
First, we assume that $\gamma$ is isotopic to one of $\alpha_i$'s ($i=1, 2, 3$). Since $\alpha$ and $\beta$ satisfy the rectangle condition, there exists a subarc of $\gamma$ connecting $\beta_i$ and $\beta_j$ whose interior is disjoint from $\be$.

Now we assume that $\gamma$ is not isotopic to any of $\alpha_i$'s ($i=1, 2, 3$). By Lemma~\ref{wave existence}, there exists a wave $\omega_\gamma$ of $\gamma$ with respect to $\alpha$. Without loss of generality, the endpoints of $w_\gamma$ lie on $\alpha_1$ as shown in Figure~\ref{Rectangle-3}a. Then $\omega_\gamma$, together with a subarc of $\alpha_1$, separates $\alpha_2$ and $\alpha_3$. Since $\alpha$ and $\beta$ satisfy a rectangle condition, there exist three pairs of adjacent subarcs of $\beta_1$ and $\beta_2$, $\beta_2$ and $\beta_3$, and $\beta_3$ and $\beta_1$, all of which connect $\alpha_2$ and $\alpha_3$. Since $\omega_\gamma$ separates $\alpha_2$ and $\alpha_3$, for every pair $\{i, j\}$ ($i\neq j$) of $\{1, 2, 3\}$ there exists a subarc of $\gamma$ connecting $\beta_i$ and $\beta_j$ whose interior is disjoint from $\beta_1\cup\beta_2\cup\beta_3$. Figure~\ref{Rectangle-3}a
illustrates these subarcs of $\gamma$ whose endpoints are $p_1$ and $p_2$, $p_3$ and $p_4$, and $p_5$ and $p_6$.

If $\eta$ is an essential arc based at $\be_i$, then $\eta$ separates $\beta_j$ and $\beta_k$. Since there is a subarc of $\gamma$ connecting $\beta_j$ and $\be_k$, $\eta$ must intersect this subarc of $\gamma$ as shown in Figure~\ref{Rectangle-3}b. This proves the second statement of the lemma.
\end{proof}

%here
We are now ready to prove Theorem~\ref{main result3}.

\begin{proof}[The proof of Theorem~\emph{\ref{main result3}}]
Since the knot $8_5$ is known to be a prime knot, it follows immediately from Theorem~\ref{main result-intro3} that the $3$-bridge decomposition $(T_1, T_{\epsilon};S)$ of the knot $8_5$ in Figure~\ref{Rectangle-7}a is strongly irreducible.

For the proof concerning the rectangle condition, we will show that there are no systems of bridge arcs for $T_1$ and $T_\ep$ which satisfy the rectangle condition.
Let $\alpha=\{\alpha_1,\alpha_2,\alpha_3\}$ and $\beta=\{\beta_1,\beta_2,\beta_3\}$ be systems of bridge arcs for $T_1$ and $T_\ep$ respectively. Then we will show that $\alpha$ and $\beta$ do not satisfy the rectangle condition. Suppose for the contradiction that $\alpha$ and $\beta$ do satisfy the rectangle condition.

First, assume that $\beta=\{\beta_1,\beta_2,\beta_3\}$ is isotopic to $\ep=\{\ep_1, \ep_2, \ep_3\}$. Then by Lemma~\ref{any bridge arc connects},
any bridge arc for $T_1$ has a subarc connecting $\ep_i$ and $\ep_j$ for every pair $\{i, j\}(i\neq j)$ in $\{1, 2, 3\}$ whose interior is disjoint from $\ep_1\cup\ep_2\cup\ep_3$.
However, the green and blue bridge arcs $\del_2$ and $\del_3$ of $T_1$ in Figure~\ref{Rectangle-7}b don't have a subarc connecting $\ep_1$ and $\ep_3$ without intersecting $\ep_2$, a contradiction.

Second, assume that $\beta$ is not isotopic to $\ep$.
It follows from Theorem~\ref{unique normal form}
that $\beta$ is not in a normal form with respect to $\ep$, implying that
there exist two adjacent intersection points $p$ and $q$ of $\beta_{j_0}$ in $\ep_i$ for some $i, j_0\in\{1, 2, 3\}$.
Let $\eta$ be the subarc of $\ep_i$ connecting $p$ and $q$.
Since bridge arcs are assumed to intersect each other transversely and minimally, $\eta$ is an essential arc based at $\beta_{j_0}$.
By Lemma~\ref{any bridge arc connects} any bridge arc for $T_1$ intersects $\eta$.
Therefore all of the red, green, blue bridge arcs $\del_1, \del_2, \del_3$ of $T_1$ in Figure~\ref{Rectangle-7} also intersect $\eta$.

Now we consider three cases: (1) $\eta\subseteq \ep_1$, (2) $\eta\subseteq \ep_2$, and (3) $\eta\subseteq \ep_3$.
In each case, we bring out a contradiction, thereby proving the theorem.

\begin{figure}[tbp]
\centering
\includegraphics[width = 1\textwidth]{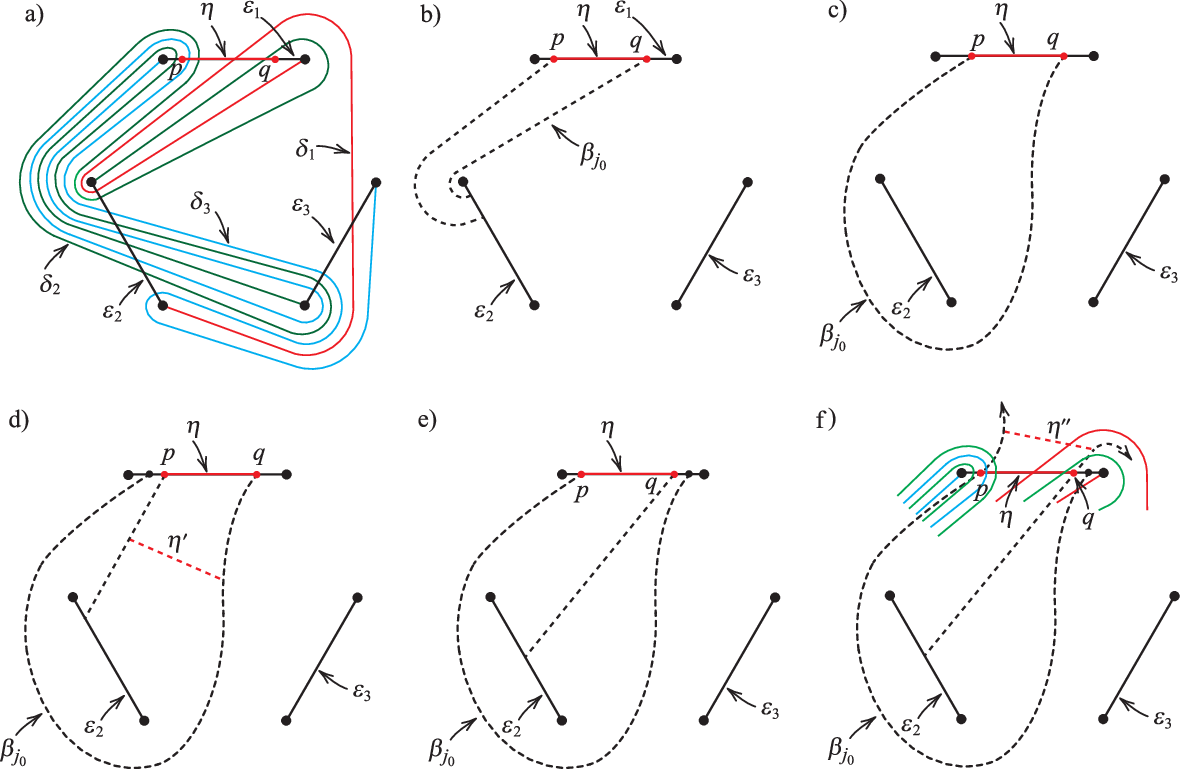}
\caption{The case where $\eta\subseteq \ep_1$ and the three cases (i), (ii), and (iii) for the subarcs of $\be_{j_0}$ with the endpoints $p$ and $q$.}
\label{Rectangle-12}
\end{figure}

(1) Suppose that $\eta\subseteq \ep_1$. Since $\eta$ intersects each of the red, green, blue bridge arcs $\del_1, \del_2, \del_3$,
the endpoints $p$ and $q$ of $\eta$ must have the configuration illustrated in Figure~\ref{Rectangle-12}a, where
$\eta$ contains at least one intersection point with each of $\del_1, \del_2, \del_3$.
By Lemma~\ref{adjacent intersections}, this situation gives rise to three possible cases:\\

\noindent(i) $p$ and $q$ are the endpoints of parallel subarcs of $\be_{j_0}$ connecting $\ep_1$ and $\ep_j$ for some $j\in\{1, 2, 3\}$;\\
(ii) $p$ and $q$ are the endpoints of a wave of $\be_{j_0}$ based at $\ep_{1}$; or\\
(iii) one of $p$ and $q$ is an endpoint of a wave of $\be_{j_0}$ based at $\ep_{1}$, and the other point is an endpoint of a subarc of $\be_{j_0}$ connecting $\ep_{1}$ and $\ep_j$ for some $j\in\{2, 3\}$. \\

\noindent However it is easy to see from Figure~\ref{Rectangle-12}a that if $\ep_j$ is $\ep_1$ or $\ep_3$ for the case (i), if the wave of $\be_{j_0}$ based at $\ep_{1}$ together with $\eta$ contains $\ep_3$ inside for the case (ii), and if $\ep_j=\ep_3$ for the case (iii), then there exists an essential arc based at $\be_{j_0}$ which doesn't intersect one of $\del_1, \del_2, \del_3$. This contradicts to Lemma~\ref{any bridge arc connects}. Therefore, in cases (i) and (iii) we must have $\ep_j=\ep_2$, and in case (ii) the wave of $\beta_{j_0}$ based at $\ep_1$ together with $\eta$ must contain $\ep_2$ inside.

%here
Now we may assume that the cases (i), (ii), and (iii) correspond to the configuration shown in Figures~\ref{Rectangle-12}b $\sim$ \ref{Rectangle-12}e respectively. Figures~\ref{Rectangle-12}b and \ref{Rectangle-12}c illustrate, respectively, a specific type of parallel subarcs and a wave of $\beta_{j_0}$ in cases (i) and (ii). Figures~\ref{Rectangle-12}d and \ref{Rectangle-12}e depict specific types of subarcs of $\be_{j_0}$ for case (iii) where two subcases arise depending on which point of $p$ and $q$ is an endpoint of the wave of $\be_{j_0}$.

By the standard property of properly embedded arcs in a pair of pants, every other type of subarc or wave of $\beta_{j_0}$ in cases (i), (ii), and (iii) can be obtained from the specific configurations shown in Figures~\ref{Rectangle-12}b $\sim$ \ref{Rectangle-12}e by performing a multiple of half Dehn twists along the disk in the bridge sphere $S$ that is the regular neighborhood of one of the bridge arcs $\ep_1, \ep_2,$ or $\ep_3$. For instance, in case (i), any parallel subarcs of $\be_{j_0}$ with endpoints $p$ and $q$ connecting $\ep_1$ and $\ep_2$ can be obtained from the parallel subarcs in Figure~\ref{Rectangle-12}b by
twisting along the disks that are regular neighborhoods of $\ep_1$ and $\ep_2$ in $S$.

For case (i), note that if we twist the parallel subarcs of $\be_{j_0}$ in Figure~\ref{Rectangle-12}b
along the disk that is the regular neighborhood of $\ep_1$ or $\ep_2$ in $S$, then we can find an essential arc $\eta'$
based at $\be_{j_0}$ , for example as illustrated in Figure~\ref{Rectangle-12-1}a, which doesn't intersect one of the bridge arcs $\del_1, \del_2,$ and $\del_3$. This contradicts Lemma~\ref{any bridge arc connects}. For the parallel subarcs of $\be_{j_0}$ in Figure~\ref{Rectangle-12}b, we can extend them as shown in Figure~\ref{Rectangle-12-1}b to eventually obtain an essential arc $\eta'$ based at $\be_{j_0}$, again leading to a contradiction with Lemma~\ref{any bridge arc connects}.

\begin{figure}[tbp]
\centering
\includegraphics[width = 0.8\textwidth]{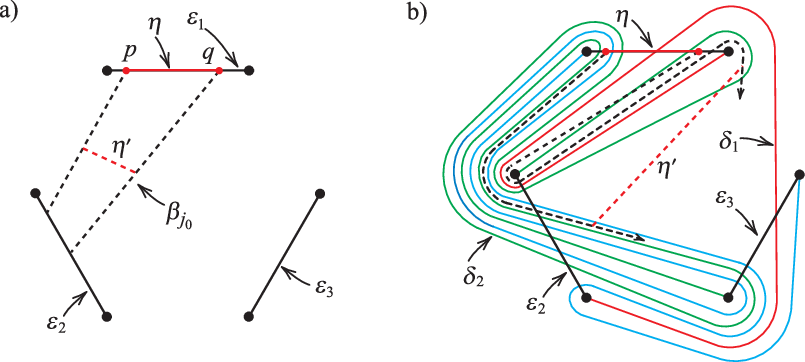}
\caption{The case (i) for the parallel subarcs of $\be_{j_0}$ in Figure~\ref{Rectangle-12}b.}
\label{Rectangle-12-1}
\end{figure}

For case (ii), one of the other two bridge arcs $\be_{i_0}$ and $\be_{k_0}$, say $\be_{i_0}$, in $\be=\{\be_1, \be_2, \be_3\}$ must be the bridge arcs $\ep_2$
because the endpoints $p$ and $q$ are adjacent intersection points of $\be$ with $\ep$.
We can observe from Figure~\ref{Rectangle-12}a that
every subarc of $\del_3$ connecting $\be_{i_0}=\ep_2$ and $\be_{k_0}$ intersects $\be_{j_0}$, which contradicts Lemma~\ref{any bridge arc connects}.

For case (iii), there are two configurations as shown in Figures~\ref{Rectangle-12}d and \ref{Rectangle-12}e. We can observe from Figure~\ref{Rectangle-12}d
that since $p$ and $q$ are adjacent points, there exists an essential arc $\eta'$ based at $\be_{j_0}$ which doesn't intersect the blue bridge arc $\del_3$, which contradicts Lemma~\ref{any bridge arc connects}. For the configuration in Figure~\ref{Rectangle-12}e, we extend the bridge arc $\be_{j_0}$ from the endpoints $p$ and $q$ without intersecting the wave of $\beta_{j_0}$. In particular, since there is a wave based at $\be_{j_0}$, the extension of $\be_{j_0}$ from the endpoint $p$ escapes from the green and blue bridge arcs $\del_2$ and $\del_3$ near $p$. Then we can construct an essential arc $\eta''$ based at $\be_{j_0}$ as illustrated in Figure~\ref{Rectangle-12}f, which doesn't intersect $\del_3$, again contradicting Lemma~\ref{any bridge arc connects}.

\begin{figure}[tbp]
\centering
\includegraphics[width = 1\textwidth]{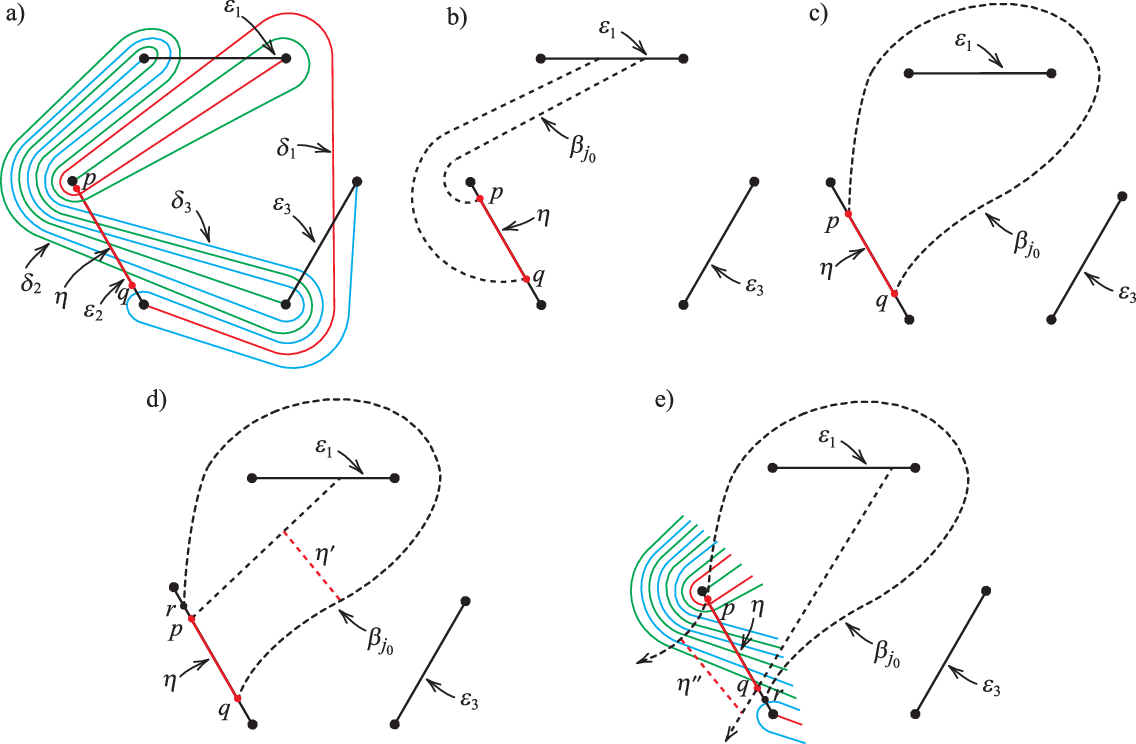}
\caption{The case where $\eta\subseteq \ep_2$ and the three cases (i), (ii), and (iii) for the subarcs of $\be_{j_0}$ with the endpoints $p$ and $q$.}
\label{Rectangle-13}
\end{figure}

(2) Suppose that $\eta\subseteq \ep_2$. We can apply an argument similar to that for the case $\eta\subseteq \ep_1$.
Since $\eta$ intersects all of the red, green, blue bridge arcs $\del_1, \del_2, \del_3$,
the endpoints $p$ and $q$ of $\eta$ have the configuration shown in Figure~\ref{Rectangle-13}a, where
$\eta$ contains at least one intersection point with each of $\del_1, \del_2, \del_3$.

There are three cases (i), (ii), and (iii) for the endpoints $p$ and $q$ as described in Lemma~\ref{adjacent intersections}:\\

\noindent(i) $p$ and $q$ are the endpoints of parallel subarcs of $\be_{j_0}$ connecting $\ep_2$ and $\ep_j$ for some $j\in\{1, 2, 3\}$;\\
(ii) $p$ and $q$ are the endpoints of a wave of $\be_{j_0}$ based at $\ep_{2}$; or\\
(iii) one of $p$ and $q$ is an endpoint of a wave of $\be_{j_0}$ based at $\ep_{2}$, and the other point is an endpoint of a subarc of $\be_{j_0}$ connecting $\ep_{2}$ and $\ep_j$ for some $j\in\{1, 3\}$. \\

As before, we can see that if $\ep_j$ is $\ep_2$ or $\ep_3$ in case (i), and if $\ep_j$ is $\ep_3$ in case (iii), then there exists an essential arc based at $\be_{j_0}$ which doesn't intersect one of $\del_1, \del_2, \del_3$, which contradicts to Lemma~\ref{any bridge arc connects}. For case (ii), if the wave of $\be_{j_0}$ based at $\ep_{2}$ together with $\eta$ contains $\ep_3$ inside, then one of the other two bridge arcs $\be_{i_0}$ and $\be_{k_0}$, say $\be_{i_0}$, in $\be=\{\be_1, \be_2, \be_3\}$ must be the bridge arcs $\ep_3$ because the endpoints $p$ and $q$ are adjacent intersection points of $\be$ with $\ep$. We can observe from Figure~\ref{Rectangle-12}a that
every subarc of $\del_1$ connecting $\be_{i_0}=\ep_3$ and $\be_{k_0}$ intersects $\be_{j_0}$, which is again a contradiction to Lemma~\ref{any bridge arc connects}.

Now we may assume that cases (i), (ii), and (iii) have the configurations shown in Figures~\ref{Rectangle-13}b $\sim$ \ref{Rectangle-13}e respectively, where Figures~\ref{Rectangle-13}d and ~\ref{Rectangle-13}e depict case (iii) depending on which of the points $p$ and $q$ is an endpoint of the wave of $\be_{j_0}$. In case (i) there are adjacent intersection points of $\be$ in $\ep_1$ that belong to $\be_{j_0}$. Therefore, this case reduces to case (1), where $\eta\subseteq \ep_1$, which has already been handled. For the other cases (ii) and (iii), we can apply a similar argument as in the case where $\eta \subseteq \epsilon_1$.

(3) Suppose that $\eta\subseteq \ep_3$.
Since $\eta$ contains at least one intersection point with each of $\del_1, \del_2, \del_3$,
there are three cases (i), (ii), and (iii) for the endpoints $p$ and $q$ as described in Lemma~\ref{adjacent intersections}.
For cases (ii) and (iii), we can apply a similar argument to that used when $\eta\subseteq \ep_1$.
For case (i), unlike the cases (1) and (2), it is easy to see that there exists an essential arc based at $\be_{j_0}$ which doesn't intersect one of $\del_1, \del_2, \del_3$, a contradiction to Lemma~\ref{any bridge arc connects}.
\end{proof}

\end{document}